\documentclass[11pt]{amsart}

\date{\today}

\usepackage{amsmath, amsthm, amscd, amsfonts, amssymb, color}
\usepackage[bookmarksnumbered, colorlinks]{hyperref}

\newcommand{\la}{\langle}
\newcommand{\ra}{\rangle}

\theoremstyle{theorem}
\newtheorem{theorem}{Theorem}[section]
\newtheorem{cor}[theorem]{Corollary}
\newtheorem{prop}[theorem]{Proposition}
\newtheorem{lemma}[theorem]{Lemma}
\theoremstyle{remark}
\newtheorem{remark}[theorem]{Remark}
\newtheorem{example}[theorem]{Example}
\newtheorem{definition}[theorem]{Definition}

\date{\today}
\title{Expansions from frame coefficients with erasures}
\author[Lj. Aramba\v si\' c]
{Ljiljana Aramba\v si\' c}
\author[D. Baki\' c]
{Damir Baki\' c$^{*}$}
\thanks{$^{*}$ Corresponding author}
\address{Department of Mathematics, University of Zagreb,
Bijeni\v cka cesta 30, 10000 Zagreb, Croatia.}
\email{arambas@math.hr}
\email{bakic@math.hr}

\begin{document}

\begin{abstract}
We propose a new approach to the problem of recovering signal from frame coefficients with erasures. Such problems arise naturally from applications where some of the coefficients could be corrupted or erased during the data transmission. Provided that the erasure set satisfies the minimal redundancy condition, we construct a suitable synthesizing dual frame which enables us to perfectly reconstruct the original signal without recovering the lost coefficients. Such  dual frames which compensate for erasures are described from various viewpoints.

In the second part of the paper frames robust with respect to finitely many erasures are investigated.
We characterize all full spark frames for finite-dimensional Hilbert spaces. In particular, we show that each full spark frames is generated by a matrix whose all square submatrices are nonsingular. In addition, we provide a method for constructing totally positive matrices. Finally, we give a method, applicable to a large class of frames, for transforming general frames into Parseval ones.
\end{abstract}

\subjclass[2010]{Primary 42C15; Secondary 47A05}

\keywords{Frame, Dual frame, Erasure, Full spark, Totally positive matrix}
\maketitle

\section{Introduction}
Frames are often used in process of encoding and decoding signals. It is the redundancy property of frames that makes them robust to erasure corrupted data. A number of articles have been written on methods for reconstruction from frame coefficients with erasures and related problems.

Recall that a sequence $(x_n)_{n=1}^{\infty}$ in a Hilbert space $H$ is a \emph{frame} for $H$ if there exist positive constants $A$ and $B$, that are called frame bounds, such that
\begin{equation}\label{frame def}
A\|x\|^2\leq \sum_{n=1}^{\infty}|\la x,x_n\ra |^2\leq B\|x\|^2,\quad \forall x \in H.
\end{equation}
If $A=B$ we say that frame is tight and, in particular, if $A=B=1$ so that
\begin{equation}
\sum_{n=1}^{\infty}|\la x,x_n\ra |^2=\|x\|^2,\quad\forall x \in H,
\end{equation}
we say that $(x_n)_{n=1}^{\infty}$ is a \emph{Parseval frame}.

If only the second inequality in \eqref{frame def} is satisfied, $(x_n)_{n=1}^{\infty}$ is said to be a \emph{Bessel sequence}.


 For each Bessel sequence $(x_n)_{n=1}^{\infty}$ in $H$ one defines the \emph{analysis operator} $U: H \rightarrow \ell^2$ by $Ux=(\langle x,x_n\rangle)_{n},\,x\in H$. It is evident that $U$ is bounded. Its adjoint operator $U^*$, which is called the \emph{ synthesis operator}, is given by $U^*((c_n)_{n})=\sum_{n=1}^{\infty}c_nx_n,\,(c_n)_{n}\in \ell^2$. Moreover, if $(x_n)_{n=1}^{\infty}$ is a frame, the analysis operator $U$ is also bounded from below, the synthesis operator $U^*$ is a surjection and the product $U^*U$ (sometimes called the \emph{frame operator}) is an invertible operator on $H$. It turns out that the sequence $(y_n=(U^*U)^{-1}x_n)_{n=1}^{\infty}$ is also a frame for $H$ that is called the \emph{canonical dual} frame and satisfies the reconstruction formula
\begin{equation}\label{defcandual}
x=\sum_{n=1}^{\infty}\la x,x_n\ra y_n,\quad \forall x \in H.
\end{equation}
In general, the canonical dual is not the only frame for $H$ which provides us with the reconstruction in terms of the frame coefficients $\la x,x_n\ra$. Any frame $(v_n)_{n=1}^{\infty}$ for $H$ that satisfies
\begin{equation}\label{defaltdual}
x=\sum_{n=1}^{\infty}\la x,x_n\ra v_n,\quad \forall x \in H
\end{equation}
is called a \emph{dual frame} for $(x_n)_{n=1}^{\infty}$.

\vspace{.1in}

In the present paper we work with infinite frames for infinite-dimensional separable Hilbert spaces and our frames will be denoted as $(x_n)_n$, $(y_n)_n$, etc. Accordingly, by writing $\sum_{n=1}^{\infty}c_nx_n$,  $\sum_{n=1}^{\infty}c_ny_n$, $\ldots$ with $(c_n)_n\in \ell^2$, we will indicate that the corresponding summations consist of infinitely many terms. However, all the results that follow (including the proofs) are valid for finite frames in finite-dimensional spaces. We shall explicitly indicate whenever a specific result concerns finite frames only and such frames  will be typically denoted as $(x_n)_{n=1}^M, (y_n)_{n=1}^M,\ldots$ with $M\in \Bbb N$.

\vspace{.1in}

Frames that are not bases are overcomplete (i.e., there exist their proper subsets which are complete) and they posses infinitely many dual frames. The \emph{excess}  of a frame is defined as the greatest integer $k$ such that $k$ elements can be deleted from the frame and still leave a complete set, or $+\infty$ if there is no upper bound to the number of elements that can be removed.
A frame is called $m$-\emph{robust}, $m\in \Bbb N$, if it is a frame whenever any $m$ of its elements are removed. In particular, we say that a frame is of \emph{uniform excess} $m$ if it is a basis whenever any $m$ of its elements are removed. Frames of uniform excess for finite-dimensional spaces are called \emph{full spark frames}.

Frames were first introduced by Duffin and Schaeffer in \cite{DS}. The readers are referred to some standard references, e.g.~\cite{CKu, Ch, HL, KC} for more information about frame theory and their applications. In particular, basic results on excesses of frames can be found in \cite{BB, BB2,BCHL,H}.

\vspace{.1in}

In applications, we first compute the frame coefficients $\la x,x_n\ra$ of a signal $x$ (analyzing or encoding $x$) and then apply \eqref{defcandual} or
\eqref{defaltdual} to reconstruct (synthesizing or decoding) $x$ using a suitable dual frame. During the processing the frame coefficients or data transmission some of the coefficients could get lost. Thus, a natural question arise: how to reconstruct the original signal in a best possible way with erasure-corrupted frame coefficients? Recently many researchers have been working on different approaches to this and related problems. In particular, we refer the readers to \cite{BO, BP, CK, GKK, HS, HP, LS, LH} and references therein.

It turns out that the perfect reconstruction is possible as long as erased coefficients are indexed by a set that satisfies the \emph{minimal redundancy condition} (\cite{LS}; see Definition~\ref{defMRC} below). Most approaches assume a pre-specified dual frame and hence aim to recover the missing coefficients using the non-erased ones. Alternatively, one may try to find an alternate dual frame, depending on the set of erased coefficients, in order to compensate for errors.

Here we use this second approach. Assuming that the set of indices $E$ for which the coefficients $\la x,x_n\ra$, $n\in E$, are erased is finite and satisfies the minimal redundancy condition, we show that there exists a frame $(v_n)_n$ dual to $(x_n)_n$ such that
\begin{equation}\label{approach}
v_n=0,\quad \forall n\in E.
\end{equation}
Obviously, such an "$E^c$-supported" frame $(v_n)_n$ (with $E^c$ denoting the complement of $E$ in the index set),  enables the perfect reconstruction using \eqref{defaltdual} without knowing or recovering the lost coefficients $\la x,x_n\ra$, $n\in E$. Such dual frames are the central object of our study in the present article.

The paper is organized as follows.
In Section 2 we prove in a constructive way, for each finite set $E$ satisfying the minimal redundancy condition, the existence of a dual frame with property \eqref{approach}. The construction is enabled by a parametrization  of dual frames by oblique projections to the range of the analysis operator that is obtained in \cite{BB}. For the construction details we refer to the second and the third proof of Theorem~\ref{main}.
Moreover, Theorem~\ref{main2} provides a concrete procedure for computing the elements of such a dual frame in terms of the canonical dual. It turns out that the computation boils down to solving certain system of linear equations.

In Section 3 we obtain in Theorem~\ref{main2-cor1} another description of  the dual frame constructed in the preceding section. Then we introduce in Theorem~\ref{thm-inverse} a finite iterative algorithm for computing the elements of the constructed dual frame. Finally, in our Theorem \ref{tm-LS} we improve a result from \cite{LS} which provides us with an alternative technique for obtaining our dual.

Section 4 is devoted to results concerning frames robust to erasures. In particular, we present a simple procedure for deriving a $1$-robust Parseval frame starting from an arbitrary  Parseval frame. The second part of Section 4 is devoted to finite frames for finite dimensional spaces. In Theorem~\ref{2N} we characterize all full spark frames. It turns out that each full spark frame is generated by a matrix whose all square submatrices are nonsingular.
This opens up a question of finding methods for construction of such matrices. We provide in Theorem~\ref{constructing TP} a method for construction of infinite totally positive matrices (a subclass that consists of matrices with all minors strictly positive).
Finally, we discuss a method for obtaining a full spark Parseval frame from a general full spark frame. In Theorem~\ref{thm-PArs_spark new} we present a finite iterative algorithm which gives us, for a positive invertible operator on a finite dimensional space, an invertible (not necessarily positive) operator $R$ such that $RFR^*=I$. This gives us a method (see Corollary~\ref{thm-PArs_spark old}) which can be applied to a large class of frames for obtaining a Parseval frame from a general one.

\vspace{.1in}

At the end of this introductory section we establish the rest of our notation. The linear span of a set $X$ will be denoted by $\text{span}\,X,$ and its closure by $\overline{\text{span}}\,X$. The set of all bounded operators on a Hilbert space $H$ is denoted $\Bbb B(H)$ (or $\Bbb B(H,K)$ if two different spaces are involved). For $x,y\in H$ we denote by $\theta_{x,y}$ a rank one operator on $H$ defined by $\theta_{x,y}(v)=\langle v,y\rangle x,\,v\in H$. The null-space and range of a bounded operator $T$ will be denoted by $\text{N}(T)$ and $\text{R}(T)$, respectively. By $X\stackrel{.}{+}Y$ we will denote a direct sum of (sub)spaces $X$ and $Y$. The space of all $m\times n$ matrices will be denoted by $M_{mn};$ if $m=n$ then we write $M_n.$
Finally, we denote by $(e_n)_n$ the canonical basis in $\ell^2$.

\vspace{0.2in}


\section{Dual frames compensating for erasures}

We begin with the definition of the minimal redundancy condition as formulated in \cite{LS}.

\begin{definition}\label{defMRC}
Let $(x_n)_n$ be a frame for a Hilbert space $H$. We say that a finite set of indices $E$ satisfies the minimal redundancy condition for $(x_n)_n$ if $\overline{\text{span}}\,\{x_n:n\in E^c\}=H$.
\end{definition}

\begin{remark}\label{first remark}
If a finite set $E$ satisfies the minimal redundancy condition for a frame $(x_n)_n$ for $H$ it is a non-trivial fact, though relatively easy to prove, that the reduced sequence $(x_n)_{n\in E^c}$ is again a frame for $H$. In fact, if $H$ is finite-dimensional, there is nothing to prove since in this situation frames are just spanning sets. We refer the reader to \cite{LS} for a proof in the infinite-dimensional case.

We also note: if $E$ satisfies the minimal redundancy condition for a frame $(x_n)_n$ with the analysis operator $U$, then $E$ has the same property for all frames of the form $(Tx_n)_n$ where $T\in \Bbb B(H)$ is a surjection. In particular, this applies to the canonical dual $(y_n)_n, y_n=(U^*U)^{-1}x_n$, $n\in \Bbb N$, and to the associated Parseval frame $((U^*U)^{-\frac{1}{2}}x_n)_n$.
\end{remark}

\begin{theorem}\label{main}
Let $(x_n)_n$ be a frame for a Hilbert space $H$. Suppose that a finite set of indices $E$ satisfies the minimal redundancy  condition for $(x_n)_n$. Then there exists a frame $(v_n)_n$ for $H$ dual to $(x_n)_n$ such that $v_n=0$ for all $n\in E$.
\end{theorem}
\begin{proof}
By Remark~\ref{first remark}, $(x_n)_{n\in E^c}$ is a frame for $H$. Consider an arbitrary dual frame of $(x_n)_{n\in E^c}$, conveniently denoted by $(v_n)_{n\in E^c}$, and put $v_n=0$ for $n\in E.$ Then $(v_n)_n$ is a frame for $H$ with the desired property.
\end{proof}

As noted in the introduction, a dual frame $(v_n)_n$ with the property from the preceding theorem enables us to reconstruct perfectly a signal $x\in H$ even when the coefficients
$\langle x,x_n\rangle,n\in E,$ are lost. This is simply because the lost coefficients are irrelevant in the reconstruction formula
$x=\sum_{n=1}^{\infty}\langle x,x_{n}\rangle v_n$ since $v_n=0$ for $n\in E.$ However, the above proof is not satisfactory from the application point of view. What we really need is a constructive proof. To provide such a proof we first need a lemma.

\begin{lemma}\label{MR equivalent}
Let $(x_n)_n$ be a frame for a Hilbert space $H$ with the analysis operator $U$. Then a finite set of indices $E$ satisfies the minimal redundancy condition for $(x_n)_n$ if and only if $\textup{R}(U)\cap \text{span}\,\{e_n:n\in E\}=\{0\}$.
\end{lemma}
\begin{proof}
Each element in $\text{R}(U)\cap \,\text{span}\,\{e_n:n\in E\}$ is of the form $s=(\langle x,x_n\rangle)_n$ with $x\in H$ such that
$\langle x,x_n\rangle =0$ for all $n\notin E$. Since $U$ is an injection, $s\not =0$ if and only if the corresponding $x\in H$ has the property $x\not =0$ and $x \perp x_n$ for all $n\notin E$. Clearly, this last condition is equivalent to $x\perp \overline{\text{span}}\,\{x_n:n\notin E\}$.
\end{proof}

\begin{proof}[The second proof of Theorem~\ref{main}.] Denote by $U$ the analysis operator of $(x_n)_n$. Recall from Corollary 2.4 in \cite{BB} that all dual frames of $(x_n)_n$ are parameterized by bounded oblique projections to $\text{R}(U)$ or, equivalently, by closed direct complements of $\text{R}(U)$ in $\ell^2$. More precisely, a frame $(v_n)_n$ with the analysis operator $V$ is dual to $(x_n)_n$ if and only if $V^*$ is of the form
$V^*=(U^*U)^{-1}U^*F$ where $F\in \Bbb B(\ell^2)$ is the oblique projection to $\text{R}(U)$ parallel to some closed subspace $Y$ of $\ell^2$ such that $\ell^2=\text{R}(U)\stackrel{.}{+}Y$. In particular, the canonical dual frame $(y_n)_n$ corresponds to the orthogonal projection $F=U(U^*U)^{-1}U^*$ to $\text{R}(U)$.

Hence, to obtain a dual frame $(v_n)_n$ with the required property $v_n=0$ for $n\in E,$ we only need to find a closed direct complement $Y$ of $\text{R}(U)$ in $\ell^2$ such that $e_n\in Y$ for all $n\in E$. Then we will have
$$Fe_n=0,\quad \forall n\in E$$ and, consequently,
$$v_n=V^*e_n=(U^*U)^{-1}U^*Fe_n=0,\quad \forall n\in E.$$

Since $E$ satisfies the minimal redundancy condition for $(x_n)_n$, Lemma~\ref{MR equivalent} tells us that $\text{R}(U)\cap \text{span}\,\{e_n:n\in E\}=\{0\}$. Denote by $Z$ the orthogonal complement of
$\text{R}(U)\stackrel{.}{+} \text{span}\,\{e_n:n\in E\}$.
(Indeed, this is a closed subspace, being a sum of two closed subspaces, one of which is finite-dimensional.)
In other words, let
\begin{equation}\label{I}
\ell^2=\left(\text{R}(U)\stackrel{.}{+} \text{span}\,\{e_n:n\in E\}\right) \oplus Z.
\end{equation}
This may be rewritten in the form
\begin{equation}\label{II}
\ell^2=\text{R}(U) \stackrel{.}{+}\left(\text{span}\,\{e_n:n\in E\} \oplus Z\right).
\end{equation}
Put
\begin{equation}\label{III}
Y=\text{span}\,\{e_n:n\in E\} \oplus Z.
\end{equation}
Clearly, $Y$ is a closed direct complement of $\text{R}(U)$ in $\ell^2$ with the desired property.
\end{proof}

Although the above proof provides more insight into the construction of a dual frame $(v_n)_n$ with the desired property, we still need more concrete realization for practical purposes. Hence we provide

\begin{proof}[The third proof of Theorem~\ref{main}.] Let us keep the notations from the preceding proof. Assume, without loss of generality, that $E=\{1,2,\ldots,k\}$.
 Recall that the synthesis operator of our desired dual $(v_n)_n$ is $V^*=(U^*U)^{-1}U^*F$, so
$v_n$'s are given by
\begin{equation}\label{IV}
v_n=(U^*U)^{-1}U^*Fe_n,\quad \forall n\in \Bbb N.
\end{equation}
We want to express $(v_n)_n$ in terms of the canonical dual frame $(y_n)_n$. Recall that
\begin{equation}\label{V}
y_n=(U^*U)^{-1}U^*e_n,\quad \forall n\in \Bbb N.
\end{equation}
Let $p_n\in \text{R}(U)$ and $a_n\in \text{R}(U)^{\perp}$ be such that
\begin{equation}\label{VI}
e_n=p_n+a_n,\quad \forall n\in \Bbb N.
\end{equation}
Since $a_n\in \text{R}(U)^{\perp}=\text{N}(U^*)$, we can rewrite \eqref{V} in the form
\begin{equation}\label{VII}
y_n=(U^*U)^{-1}U^*p_n,\quad \forall  n\in \Bbb N.
\end{equation}
Recall now that $U(U^*U)^{-1}U^*$ is the orthogonal projection onto $\text{R}(U)$. Hence, by applying $U$ to \eqref{VII} we get
\begin{equation}\label{VIII}
Uy_n=p_n,\quad \forall  n\in \Bbb N.
\end{equation}
On the other hand, using \eqref{II}, we can find $r_n\in \text{R}(U), b_n \in \text{span}\,\{e_{1},e_{2},\ldots , e_{k}\}$ and $c_n\in Z$ such that
\begin{equation}\label{IX}
e_n=r_n+b_n+c_n,\quad \forall n\in \Bbb N.
\end{equation}
Since $F$ is the oblique projection to $\text{R}(U)$ along $\text{span}\,\{e_{1},e_{2},\ldots , e_{k}\} \oplus Z$, we have
\begin{equation}\label{X}
Fe_n=r_n,\quad \forall n\in \Bbb N.
\end{equation}
Observe that
\begin{equation}\label{XI}
b_n=e_n,\ r_n=0,\ c_n=0, \quad \forall n=1,2,\ldots, k.
\end{equation}
Since each $b_n$ belongs to $\text{span}\,\{e_{1},e_{2},\ldots , e_{k}\}$, there exist coefficients $\alpha_{ni}$ such that
\begin{equation}\label{XII}
b_n=\sum_{i=1}^k\alpha_{ni}e_i,\quad \forall n\in \Bbb N.
\end{equation}
Note that \eqref{XI} implies
\begin{equation}\label{XIII}
\alpha_{ni}=\delta_{ni}, \quad \forall n,i=1,2,\ldots, k.
\end{equation}
We now have for all $n\in \Bbb N$
\begin{eqnarray*}
e_n&\stackrel{\eqref{IX}}{=}&r_n+b_n+c_n\\
 &\stackrel{\eqref{XII}}{=}&r_n+\sum_{i=1}^k\alpha_{ni}e_i+c_n\\
 &\stackrel{\eqref{VI}}{=}&r_n+\sum_{i=1}^k\alpha_{ni}(p_i+a_i)+c_n\\
 &=&\left(r_n+\sum_{i=1}^k\alpha_{ni}p_i\right)+\left(\sum_{i=1}^k\alpha_{ni}a_i+c_n\right).
\end{eqnarray*}
Observe that $\left(r_n+\sum_{i=1}^k\alpha_{ni}p_i\right) \in \text{R}(U)$, while $\left(\sum_{i=1}^k\alpha_{ni}a_i+c_n\right) \in \text{R}(U)^{\perp}$. Thus, comparing this last equality with \eqref{VI} we obtain
\begin{equation}\label{XIV}
r_n=p_n-\sum_{i=1}^k\alpha_{ni}p_i,\quad a_n=\sum_{i=1}^k\alpha_{ni}a_i+c_n,\quad \forall n\in \Bbb N.
\end{equation}
Finally, we conclude that for all $n\in \Bbb N$
\begin{eqnarray}\label{XV}
v_n&\stackrel{\eqref{IV}}{=}&(U^*U)^{-1}U^*Fe_n \nonumber\\
 &\stackrel{\eqref{X}}{=}&(U^*U)^{-1}U^*r_n \nonumber\\
 &\stackrel{\eqref{XIV}}{=}&(U^*U)^{-1}U^*\left(p_n-\sum_{i=1}^k\alpha_{ni}p_i\right) \nonumber\\
 &\stackrel{\eqref{VII}}{=}&y_n-\sum_{i=1}^k\alpha_{ni}y_i.
\end{eqnarray}
Note that \eqref{XV} and \eqref{XIII} show that $v_1=v_2=\ldots =v_k=0$, as required.
\end{proof}

The preceding proof describes our desired dual frame $(v_n)_n$ in terms of the canonical dual $(y_n)_n$. Obviously, to obtain $v_n$'s one has to compute all the coefficients $\alpha_{ni},\,i=1,2,\ldots,k,\,n\geq k+1$.

To do that, let us first note the following useful consequence of the preceding computation. We claim that
\begin{equation}\label{XVI}
\langle v_n,x_i\rangle=-\alpha_{ni},\quad \forall  i=1,2,\ldots,k,\  \forall n\geq k+1.
\end{equation}
Indeed, for $i=1,2,\ldots,k$ and $n\geq k+1$ we have
\begin{eqnarray*}
\langle v_n,x_i\rangle&=&\langle v_n,U^*e_i\rangle\\
 &\stackrel{\eqref{IV}}{=}&\langle U(U^*U)^{-1}U^*Fe_n,e_i\rangle\\
 &=&\langle Fe_n,e_i\rangle\\
 &\stackrel{\eqref{X}}{=}&\langle r_n,e_i\rangle\\
 &\stackrel{\eqref{IX}}{=}&\langle e_n-b_n-c_n,e_i \rangle \quad (\mbox{since }i<n\mbox{ and }c_n\perp e_i)\\
 &=&-\langle b_n,e_i \rangle\\
 &\stackrel{\eqref{XII}}{=}&-\alpha_{ni}.
\end{eqnarray*}

For each $n\geq k+1$ we can rewrite \eqref{XVI}, using \eqref{XV}, as
$$
\left\langle y_n-\sum_{j=1}^k\alpha_{nj}y_j,x_i\right\rangle=-\alpha_{ni},\quad \forall i=1,2,\ldots, k
$$
or, equivalently,
$$
\sum_{j=1}^k\langle y_j,x_i\rangle \alpha_{nj}-\alpha_{ni}=\langle y_n,x_i\rangle,\quad \forall i=1,2,\ldots, k.
$$
The above equalities can be regarded as a system of $k$ equations in unknowns $\alpha_{n1},\alpha_{n2},\ldots,\alpha_{nk}$ that can be written in the matrix form as
\begin{equation}\label{XVII}
\left(\left[\begin{array}{cccc}
\langle y_1,x_1\rangle&\langle y_2,x_1\rangle&\ldots&\langle y_k,x_1\rangle\\
\langle y_1,x_2\rangle&\langle y_2,x_2\rangle&\ldots&\langle y_k,x_2\rangle\\
\vdots&\vdots& &\vdots\\
\langle y_1,x_k\rangle&\langle y_2,x_k\rangle&\ldots&\langle y_k,x_k\rangle
\end{array}
\right]-I\right)
\left[
\begin{array}{c}
\alpha_{n1}\\
\alpha_{n2}\\
\vdots\\
\alpha_{nk}
\end{array}
\right]=
\left[
\begin{array}{c}
\langle y_n,x_1\rangle\\
\langle y_n,x_2\rangle\\
\vdots\\
\langle y_n,x_k\rangle
\end{array}
\right],
\end{equation}
where $I$ denotes the unit $k\times k$ matrix. Note that the matrix of the above system is independent of $n$.

The following lemma should be compared to Lemma~\ref{MR equivalent}. It gives us another condition equivalent to the minimal redundancy property.

\begin{lemma}\label{crossGramian}
Let $(x_n)_n$ be a frame for a Hilbert space $H$ with the canonical dual $(y_n)_n$. Then a finite set of indices $E=\{n_1,n_2,\ldots,n_k\},k\in\Bbb N,$ satisfies the minimal redundancy condition for $(x_n)_n$ if and only if
$$
\left[\begin{array}{cccc}
\langle y_{n_1},x_{n_1}\rangle&\langle y_{n_2},x_{n_1}\rangle&\ldots&\langle y_{n_k},x_{n_1}\rangle\\
\langle y_{n_1},x_{n_2}\rangle&\langle y_{n_2},x_{n_2}\rangle&\ldots&\langle y_{n_k},x_{n_2}\rangle\\
\vdots&\vdots& &\vdots\\
\langle y_{n_1},x_{n_k}\rangle&\langle y_{n_2},x_{n_k}\rangle&\ldots&\langle y_{n_k},x_{n_k}\rangle
\end{array}
\right]-I
$$
is an invertible matrix.
\end{lemma}
\begin{proof}
We can assume without loss of generality that $E=\{1,2,\ldots,k\}$.

Let us first prove the lemma under additional hypothesis that $(x_n)_n$ is a Parseval frame. Note that then $(x_n)_n$ coincides with its canonical dual, i.e.~$y_n=x_n$ for all $n\in \Bbb N$. Denote $H_E=\text{span}\,\{x_1,x_2,\ldots,x_k\}.$ Then $(x_n)_{n=1}^k$ is, as a spanning set, a frame for $H_E$; let $U_E\in \Bbb B(H_E,\Bbb C^k)$ be its analysis operator.

By definition, $E$ does not satisfy the minimal redundancy condition for $(x_n)_n$ if and only if there exists $h\in H,\,h\neq 0$, such that $h \perp x_n$ for all $n\geq k+1.$ By the Parseval property of $(x_n)_n$ this is equivalent with the existence of $h\neq 0$ such that $h=\sum_{n=1}^k\langle h,x_n\rangle x_n,$ i.e. $h\in H_E\setminus\{0\}$ such that $h=U_E^*U_E h.$ This means that 1 belongs to the spectrum of $U_E^*U_E$ or, equivalently, that 1 belongs to the spectrum of $U_EU_E^*.$
So, $E$ does not satisfy the minimal redundancy condition for $(x_n)_n$ if and only if
$U_EU_E^*-I$ is a non-invertible operator. Since the matrix of $U_EU_E^*-I$ in the canonical basis of $\Bbb C^k$ is precisely $$\left[\begin{array}{cccc}
\langle x_1,x_1\rangle&\langle x_2,x_1\rangle&\ldots&\langle x_k,x_1\rangle\\
\langle x_1,x_2\rangle&\langle x_2,x_2\rangle&\ldots&\langle x_k,x_2\rangle\\
\vdots&\vdots& &\vdots\\
\langle x_1,x_k\rangle&\langle x_2,x_k\rangle&\ldots&\langle x_k,x_k\rangle
\end{array}
\right]-I,
$$
the statement is proved in the Parseval case.

We now prove the general case. Denote the analysis operator of $(x_n)_n$ by $U$. Recall that $y_n=(U^*U)^{-1}x_n,\,n\in \Bbb N$. We shall also need the associated Parseval frame $(p_n)_n$, where $p_n=(U^*U)^{-\frac{1}{2}}x_n$ for all $n\in \Bbb N$. Observe that, since $(U^*U)^{-\frac{1}{2}}$ is an invertible operator, the set $E$ satisfies the minimal redundancy condition for $(x_n)_n$ if and only if it satisfies the same condition for $(p_n)_n$. Since $(p_n)_n$ is a Parseval frame, by the first part of the proof $E$ satisfies the minimal redundancy condition for $(p_n)_n$ if and only if
$$\left[\begin{array}{cccc}
\langle p_1,p_1\rangle&\langle p_2,p_1\rangle&\ldots&\langle p_k,p_1\rangle\\
\langle p_1,p_2\rangle&\langle p_2,p_2\rangle&\ldots&\langle p_k,p_2\rangle\\
\vdots&\vdots& &\vdots\\
\langle p_1,p_k\rangle&\langle p_2,p_k\rangle&\ldots&\langle p_k,p_k\rangle
\end{array}
\right]-I
$$
is an invertible matrix.
To conclude the proof one only needs to observe that $$\langle p_i,p_j\rangle=\langle (U^*U)^{-\frac{1}{2}}x_i,(U^*U)^{-\frac{1}{2}}x_j\rangle =\langle (U^*U)^{-1}x_i,x_j\rangle=\langle y_i,x_j\rangle$$ for all $i$ and $j$.
\end{proof}

\vspace{.1in}

Let us now turn back to the system \eqref{XVII}. By the preceding lemma the matrix of the system is invertible; hence, the system has a unique solution $(\alpha_{n1},\alpha_{n2},\ldots,\alpha_{nk})$ for each $n\geq k+1$.
Thus, we have finally obtained a concrete description of our desired frame $(v_n)_n$.

\begin{theorem}\label{main2}
Let $(x_n)_n$ be a frame for a Hilbert space $H$ with the canonical dual $(y_n)_n$. Suppose that a finite set of indices $E=\{n_1,n_2,\ldots,n_k\},k\in\Bbb N,$ satisfies the minimal redundancy condition for $(x_n)_n$. For each $n \in E^c$ let
$(\alpha_{n1},\alpha_{n2},\ldots,\alpha_{nk})$ be a (unique) solution of the system
\begin{equation}\label{XVIII}
\left(\left[\begin{array}{cccc}
\langle y_{n_1},x_{n_1}\rangle&\langle y_{n_2},x_{n_1}\rangle&\ldots&\langle y_{n_k},x_{n_1}\rangle\\
\langle y_{n_1},x_{n_2}\rangle&\langle y_{n_2},x_{n_2}\rangle&\ldots&\langle y_{n_k},x_{n_2}\rangle\\
\vdots&\vdots& &\vdots\\
\langle y_{n_1},x_{n_k}\rangle&\langle y_{n_2},x_{n_k}\rangle&\ldots&\langle y_{n_k},x_{n_k}\rangle
\end{array}
\right]-I\right)
\left[
\begin{array}{c}
\alpha_{n1}\\
\alpha_{n2}\\
\vdots\\
\alpha_{nk}
\end{array}
\right]=
\left[
\begin{array}{c}
\langle y_n,x_{n_1}\rangle\\
\langle y_n,x_{n_2}\rangle\\
\vdots\\
\langle y_n,x_{n_k}\rangle
\end{array}
\right].
\end{equation}
Put
\begin{equation}\label{XIX}
v_{n_1}=v_{n_2}=\ldots =v_{n_k}=0,\quad v_n=y_n-\sum_{i=1}^k\alpha_{ni}y_{n_i},\quad n\not =n_1,n_2,\ldots, n_k.
\end{equation}
Then $(v_n)_n$ is a frame for $H$ dual to $(x_n)_n$.
\end{theorem}

\vspace{.1in}

\begin{remark}
(a) Clearly, if $(x_n)_n$ is a Parseval frame, our constructed dual frame $(v_n)_n$ is expressed in terms of the original frame members $x_n$'s.

(b) Note that the matrix of the system \eqref{XVIII} is independent not only of $n$, but also of all $x\in H$. Thus, the inverse matrix can be computed in advance, without knowing for which $x$ the coefficients $\langle x,x_{n_1}\rangle$, $\langle x,x_{n_2}\rangle, \ldots, \langle x,x_{n_k}\rangle$ will be lost.

(c) The existence of frames dual to $(x_n)_n$ with the property as in Theorem~\ref{main} is also proved in \cite{LS} (Theorem 5.2) and \cite{HS} (Theorem 2.5). However, in both of these papers such dual frames are discussed only for finite frames in finite-dimensional spaces. Besides, both papers focus on recovering the lost coefficients, rather than alternating a dual frame. We also note that the special case of our Lemma~\ref{crossGramian} (namely, the same statement for finite frames) is proved in Lemma~2.3 from \cite{HS} using a different technique.
\end{remark}

\vspace{.1in}

Let us now turn to the examples.

\begin{example}\label{ex1}
Let $(\epsilon_n)_n$ be an orthonormal basis for a Hilbert space $H$. Let
$$
x_1=\frac{1}{3}\epsilon_1,\ x_2=\frac{2}{3}\epsilon_1-\frac{1}{\sqrt{2}}\epsilon_2, \  x_3=\frac{2}{3}\epsilon_1+\frac{1}{\sqrt{2}}\epsilon_2, x_n=\epsilon_{n-1},\ \forall n\ge 4.
$$
One easily checks that $(x_n)_n$ is a Parseval frame for $H$; thus, here we have $y_n=x_n$ for all $n\in \Bbb N$. It is obvious that the set $E=\{1\}$ has the minimal redundancy property for $(x_n)_n$.
We shall use Theorem~\ref{main2} to construct a dual frame $(v_n)_n$ for $(x_n)_n$ such that $v_1=0$.

In this situation the system \eqref{XVIII} from Theorem~\ref{main2} reduces to a single equation
$$
-\frac{8}{9}\alpha_{n1}=\langle x_n,x_1\rangle,\quad \forall n\geq 2.
$$
Since $\langle x_2,x_1\rangle =\frac{2}{9}$, $\langle x_3,x_1\rangle =\frac{2}{9}$, and $\langle x_n,x_1\rangle =0$ for $n\geq 4$, we find
$\alpha_{21}=\alpha_{31}=-\frac{1}{4}$ and $\alpha_{n1}=0$ for $n\geq 4.$
Using \eqref{XV} we finally get
$$
v_1=0,\,v_2=x_2+\frac{1}{4}x_1,\,v_3=x_3+\frac{1}{4}x_1,\,v_n=x_n,\ \forall n\ge 4,
$$
that is,
$$
v_1=0,
 v_2=\frac{3}{4}\epsilon_1-\frac{1}{\sqrt{2}}\epsilon_2,  v_3=\frac{3}{4}\epsilon_1+\frac{1}{\sqrt{2}}\epsilon_2, v_n=\epsilon_{n-1}, \ \forall n\ge 4,
$$
\end{example}

\vspace{.1in}

\begin{example}\label{ex2}
Let $(\epsilon_1,\epsilon_2)$ be an orthonormal basis of a $2$-dimensional Hilbert space $H$. Consider a frame $(x_n)_{n=1}^4$ where
$$
x_1=\frac{1}{2}\epsilon_1,\,x_2=\frac{1}{2}\epsilon_2,\,x_3=\frac{1}{2}\epsilon_1-\frac{1}{2}\epsilon_2,\,x_4=\frac{1}{2}\epsilon_1+\frac{1}{2}\epsilon_2.
$$
One easily verifies that $(x_n)$ is a tight frame with $U^*U=\frac{3}{4}I,$ so the members of the canonical dual $(y_n)_{n=1}^4$ are given by
$$
y_1=\frac{2}{3}\epsilon_1,\,y_2=\frac{2}{3}\epsilon_2,\,y_3=\frac{2}{3}\epsilon_1-\frac{2}{3}\epsilon_2,\,y_4=\frac{2}{3}\epsilon_1+\frac{2}{3}\epsilon_2.
$$
Obviously, the set $E=\{1,2\}$ satisfies the minimal redundancy condition for $(x_n)_{n=1}^4$. To obtain the dual frame $(v_n)_{n=1}^4$ from Theorem~\ref{main2}
we have to solve the system
$$
\left(\left[\begin{array}{cc}\la y_1,x_1\ra&\la y_2,x_1\ra\\\la y_1,x_2\ra&\la y_2,x_2\ra\end{array}\right]-I\right)\left[\begin{array}{c}\alpha_{n1}\\\alpha_{n2}\end{array}\right]=
\left[\begin{array}{c}\la y_n,x_1\ra\\\la y_n,x_2\ra\end{array}\right],\quad n=3,4.
$$
Since
$$
\left[\begin{array}{cc}\la y_1,x_1\ra&\la y_2,x_1\ra\\\la y_1,x_2\ra&\la y_2,x_2\ra\end{array}\right]-I=\left[\begin{array}{rr}-\frac{2}{3}&0\\0&-\frac{2}{3}
\end{array}
\right],
$$
we have
$$
\left[\begin{array}{c}\alpha_{n1}\\
\alpha_{n2}\end{array}\right]=\left[\begin{array}{rr}-\frac{3}{2}&0\\
0&-\frac{3}{2}\end{array}
\right]
\left[\begin{array}{c}\la y_n,x_1\ra\\\la y_n,x_2\ra\end{array}\right],\quad n=3,4.
$$
From this one easily finds
$$
\alpha_{31}=-\frac{1}{2},\,\alpha_{32}=\frac{1}{2},\,\alpha_{41}=-\frac{1}{2},\,\alpha_{42}=-\frac{1}{2}.
$$
Theorem~\ref{main2} gives us now
$$
v_1=0,\,v_2=0,\, v_3=y_3+\frac{1}{2}y_1-\frac{1}{2}y_2,\,v_4=y_4+\frac{1}{2}y_1+\frac{1}{2}y_2,
$$
that is,
$$
v_1=0,\,v_2=0,\,v_3=\epsilon_1-\epsilon_2,\,v_4=\epsilon_1+\epsilon_2.
$$
\end{example}

\vspace{.1in}

We end this section with the proposition that provides another characterization for the minimal redundancy property of a finite set of indices.

\begin{prop}\label{some dual}
Let $(x_n)_n$ be a frame for a Hilbert space $H$, let $E=\{n_1,n_2,\ldots, n_k\}$, $k\in \Bbb N$ be a finite set of indices. Then $E$ satisfies  the minimal redundancy condition for $(x_n)_n$ if and only if there exists a frame $(z_n)_n$ for $H$ dual to $(x_n)_n$ such that
\begin{equation}\label{matrica_x_z}
\left[\begin{array}{cccc}
\langle z_{n_1},x_{n_1}\rangle&\langle z_{n_2},x_{n_1}\rangle&\ldots&\langle z_{n_k},x_{n_1}\rangle\\
\langle z_{n_1},x_{n_2}\rangle&\langle z_{n_2},x_{n_2}\rangle&\ldots&\langle z_{n_k},x_{n_2}\rangle\\
\vdots&\vdots& &\vdots\\
\langle z_{n_1},x_{n_k}\rangle&\langle z_{n_2},x_{n_k}\rangle&\ldots&\langle z_{n_k},x_{n_k}\rangle
\end{array}
\right]-I
\end{equation}
is an invertible matrix.
\end{prop}
\begin{proof} If $E$ satisfies the minimal redundancy condition for $(x_n)_n$ than, by Lemma~\ref{crossGramian}, the canonical dual frame of $(x_n)$ can be taken for $(z_n).$

Let us prove the converse. We again assume that $E=\{1,2,\ldots ,k\}$. Denote by $W$ the analysis operator of $(z_n)_n$. Recall that $W^*=(U^*U)^{-1}U^*Q$ where $Q\in \Bbb B(\ell^2)$ is some oblique projection onto $\text{R}(U)$. Observe that  $UW^*=U(U^*U)^{-1}U^*Q=Q$, since $U(U^*U)^{-1}U^*$ is the orthogonal projection to $\text{R}(U)$. Notice also that $UW^*e_j=(\langle z_j,x_n\rangle)_n$ for each $j\in \Bbb N$.

Let $P_k\in \Bbb B(\ell^2)$ denote the orthogonal projection to $\text{span}\,\{e_1,e_2,\ldots,e_k\}.$
Then $\langle P_kUW^*P_ke_j,e_n\rangle=\langle W^*e_j,U^*e_n\rangle = \langle z_j,x_n\rangle$ for all $i,n=1,\ldots,k,$ so the matrix of the compression $P_kUW^*P_k$ of $UW^*$ to $\text{span}\,\{e_1,e_2,\ldots,e_k\}$ with respect to the basis $(e_n)_{n=1}^k$ has the form
$$
[P_kUW^*P_k]_{(e_n)_{n=1}^k}=
\left[\begin{array}{cccc}
\langle z_1,x_1\rangle&\langle z_2,x_1\rangle&\ldots&\langle z_k,x_1\rangle\\
\langle z_1,x_2\rangle&\langle z_2,x_2\rangle&\ldots&\langle z_k,x_2\rangle\\
\vdots&\vdots& &\vdots\\
\langle z_1,x_k\rangle&\langle z_2,x_k\rangle&\ldots&\langle z_k,x_k\rangle
\end{array}
\right].
$$

Suppose  that $E$ does not satisfy the minimal redundancy condition. By Lemma~\ref{MR equivalent} there is $h\in H$ such that
$Uh=(\langle h,x_1\rangle, \ldots ,\langle h,x_k\rangle,0,0,\ldots)\not =0$. Since $Q$ is a projection on $\text{R}(U)$ and $Uh\in \text{R}(U)$, we have
$$
UW^*Uh=QUh=Uh.
$$
Since $Uh \in \text{span}\,\{e_1,e_2,\ldots,e_k\}$, the last equality can be written as
$$
P_kUW^*P_kUh=Uh.
$$
Thus, 1 is in the spectrum of  $P_kUW^*P_k$ which means that $[P_kUW^*P_k]_{(e_n)_{n=1}^k}-I$ is a non-invertible matrix - a contradiction.
\end{proof}

\vspace{.1in}

\begin{remark}
In the light of the preceding proposition and Lemma~\ref{crossGramian} one may ask: if the set $E=\{n_1,n_2,\ldots, n_k\}$ satisfies the minimal redundancy condition for a frame $(x_n)_n$ can we conclude that the matrix from \eqref{matrica_x_z}
is invertible for \emph{each} dual frame $(z_n)_n$?

The answer is negative which is demonstrated by the following simple example. Take an orthonormal basis $\{\epsilon_1,\epsilon_2\}$ of a two-dimensional space $H$ and consider a frame $(x_n)_{n=1}^3=(\epsilon_1, \epsilon_1+\epsilon_2,\epsilon_2)$ and its dual $(\epsilon_1,0,\epsilon_2)$. Then the set $E=\{1\}$ satisfies the minimal redundancy condition for $(x_n)_{n=1}^3$ but the corresponding matrix from the preceding proposition is equal to $0$.

Let us also mention that in general, if $E$ satisfies the minimal redundancy property for $(x_n)_n,$ there are dual frames to $(x_n)_n$ different from the canonical dual frame for which the matrix from \eqref{matrica_x_z} is invertible. For example, this matrix is invertible for every dual $(v_n)_n$ as in Theorem~\ref{main}, i.e. such that $v_n=0$ for all $n\in E$ (and, unless $(x_n)_n$ is of a very special form, such dual frames differ from the canonical dual frame associated to $(x_n)_n$).
\end{remark}

\vspace{0.2in}


\section{Computing a dual frame optimal for erasures}
In this section we give an alternative way for computing the constructed dual.
We first show in Theorem~\ref{main2-cor1} that the dual frame from Theorem~\ref{main2} can be written in another form. That naturally leads to the question of inverting operators of the form $I-G$, where $G$ has a finite rank.
In practice, $G$ will be of the form $G=\sum_{i=1}^k\theta{y_{n_i},x_{n_i}},$ where $x_{n_i}$ are frame members indexed by an erasure set $E=\{n_1,n_2,\ldots,n_k\}$, and $y_{n_i}$ are the corresponding members of the canonical dual.  (Recall from the introduction that $\theta_{y,x}$ denotes a rank one operator defined by $\theta_{y,x}(v)=\langle v,x\rangle y,\,v\in H$.) Here we provide a (finite) iterative process for computing $(I-\sum_{i=1}^k\theta{y_{n_i},x_{n_i}})^{-1}$  with the number of iterations equal to $k$.
At the end of this section we show that this inverse operator can also be computed
by applying the formula given in Theorem~6.2 in \cite{LS}. This theorem is proved for invertible operators of the form $I-\sum_{i=1}^k\theta_{y_i,x_i},$ where $y_1,\ldots,y_k$ are linearly independent. Since we cannot expect from our frame elements $y_{n_1},\ldots,y_{n_k}$ to be linearly independent, in order to apply this theorem to our situation, we prove in our Theorem \ref{tm-LS} that the conclusion of Theorem~6.2 from \cite{LS} holds even without the linear independence assumption.

\vspace{.1in}

Consider an arbitrary frame $(x_n)_n$ for $H$ with the analysis operator $U$ and a finite set of indices $E=\{n_1,n_2,\ldots , n_k\}$. Obviously, sequences $(x_n)_{n\in E^c}$ and $(x_n)_{n\in E}$ are Bessel. ($E^c$ denotes the complement of $E$ in the index set.) Denote the corresponding analysis operators by $U_{E^c}$ and $U_E$, respectively. Notice that $(x_n)_{n\in E}$ is finite, so $U_E$ takes values in $\Bbb C^k$. It is evident that the corresponding frame operators satisfy $U_{E^c}^*U_{E^c}x=\sum_{n\in E^c}\la x,x_n\ra x_n$, $U_{E}^*U_{E}x=\sum_{n\in E}\la x,x_n\ra x_n$, $x\in H$, and hence
\begin{equation}\label{the reduced frame operator}
U_{E^c}^*U_{E^c}=U^*U-U_{E}^*U_{E}.
\end{equation}

Further, if $(y_n)_n$ is the canonical dual of $(x_n)_n,$ its analysis operator $V$ is of the form $V=U(U^*U)^{-1}$. The analysis operators of Bessel sequences $(y_n)_{n\in E^c}$ and $(y_n)_{n\in E}$ will be denoted by $V_{E^c}$ and $V_E$, respectively. Observe that $V_E=U_E(U^*U)^{-1}$ and $V_{E^c}=U_{E^c}(U^*U)^{-1}$. Now the equality $x=\sum_{n=1}^{\infty}\la x,x_n\ra y_n,x\in H,$ can be written as
\begin{equation}\label{the reduced dual_operators}
V_{E^c}^*U_{E^c}=I-V_{E}^*U_{E}.
\end{equation}

\vspace{.1in}

The following lemma should be compared to Lemma~\ref{MR equivalent} and Lemma~\ref{crossGramian}. It gives us another characterization of finite sets satisfying the minimal redundancy condition.

\begin{lemma}\label{bivsa lema 2 u sec 4}
Let $(x_n)_n$ be a frame for a Hilbert space $H$ with the analysis operator $U$ and the canonical dual $(y_n)_n.$ If $E$ is any finite set of indices the following statements are mutually equivalent:
\begin{enumerate}
  \item[(a)] $E$ satisfies the minimal redundancy condition for $(x_n)_n;$
  \item[(b)] $I-V_{E}^*U_{E}$ is invertible;
  \item[(c)] $I-V_{F}^*U_{F}$ is invertible for every $F\subseteq E.$
  \end{enumerate}
\end{lemma}
\begin{proof}
Since $U_{E^c}$ is the analysis operator of a Bessel sequence $(x_n)_{n\in E^c}$,
we know that $(x_n)_{n\in E^c}$ is a frame for $H$ if and only if $U_{E^c}^*U_{E^c}$ is invertible, i.e., by \eqref{the reduced frame operator}, if and only if $U^*U-U_{E}^*U_{E}$ is invertible.
Since $U^*U$ is invertible and $$U^*U-U_{E}^*U_{E}=U^*U(I-(U^*U)^{-1}U_{E}^*U_{E})= U^*U(I-V_{E}^*U_{E}),$$ this is equivalent with invertibility of $I-V_{E}^*U_{E}.$ Thus, we have proved (a)$\Leftrightarrow$(b).

The implication (c)$\Rightarrow$(b) is obvious.

For (a)$\Rightarrow$(c) we only need to observe that, if $E$ satisfies the minimal redundancy condition for $(x_n)_n$, then so does every subset of $E;$ thus, we can apply  the implication (a)$\Rightarrow$(b) to $F.$
\end{proof}

\vspace{.1in}

Suppose that a finite set $E$ has the minimal redundancy property for a frame $(x_n)_n$ for $H$. If $(y_n)_n$ is its canonical dual, then by \eqref{the reduced dual_operators} we have
$$
(I-V_E^*U_E)x=\sum_{n\in E^c}\la x,x_n \ra y_n,\quad \forall x\in H.
$$
By the above lemma, $I-V_E^*U_E$ is an invertible operator, so we get from the preceding equality that
$$
x=\sum_{n\in E^c}\la x,x_n \ra (I-V_E^*U_E)^{-1}y_n,\quad \forall x\in H.
$$
This shows that the sequence $((I-V_E^*U_E)^{-1}y_n)_{n\in E^c}$ is a frame for $H$ that is dual to $(x_n)_{n\in E^c}$. As noted in the first proof of Theorem~\ref{main} its trivial extension (with null-vectors as frame members indexed by $n\in E$) serves as one of the frames for $H$ dual to the original frame $(x_n)_n$ that provides perfect reconstruction without coefficients indexed by the erasure set $E$.

Moreover, it turns out that $((I-V_E^*U_E)^{-1}y_n)_{n\in E^c}$ is precisely the dual frame from our construction that is given by formula \eqref{XIX} in Theorem~\ref{main2}.

\begin{theorem}\label{main2-cor1}
Let $(x_n)_n$ be a frame for a Hilbert space $H$ with the analysis operator $U$ and the canonical dual $(y_n)_n.$ Suppose that a finite set $E=\{n_1,n_2,\ldots,n_k\}$, $k\in \Bbb N$, satisfies the minimal redundancy condition for $(x_n)_n$. If $(v_n)_n$ is a dual frame for $(x_n)_n$ defined by \eqref{XIX} then
\begin{equation}\label{nova formula za veove}
v_n=(I-V_E^*U_E)^{-1}y_n,\quad  \forall  n\in E^c.
\end{equation}
\end{theorem}
\begin{proof}
Observe first that $I-V_E^*U_E$ is invertible by Lemma~\ref{bivsa lema 2 u sec 4}, so \eqref{nova formula za veove} makes sense.

Let $\alpha_n:=\begin{bmatrix}
                 \alpha_{n1} & \ldots & \alpha_{nk} \\
               \end{bmatrix}^T,n\in E^c.$
Then \eqref{XIX} can be written as
\begin{equation}\label{XIX-Pars1}
v_n=y_n-V_E^*\alpha_n,\quad \forall n\in E^c,
\end{equation}
while \eqref{XVIII} becomes
\begin{equation}\label{XVIII-Pars1}
    (U_EV_E^*-I)\alpha_n=U_Ey_n,\quad \forall n\in E^c.
\end{equation}
By Lemma~\ref{crossGramian} the operator $U_EV_E^*-I$ is invertible so \eqref{XVIII-Pars1} becomes
\begin{equation}\label{XVIII-Pars2}
    \alpha_n=(U_EV_E^*-I)^{-1}U_Ey_n,\quad \forall n\in E^c.
\end{equation}
Observe now the equality $(U_EV_E^*-I)U_E=U_E(V_E^*U_E-I)$.
Multiplying by $(U_EV_E^*-I)^{-1}$ from the left and by $(V_E^*U_E-I)^{-1}$ from the right we obtain
\begin{equation}\label{medjukorak}
(U_EV_E^*-I)^{-1}U_E=U_E(V_E^*U_E-I)^{-1}.
\end{equation}
We now have, for all $n\in E^c$,
\begin{eqnarray*}
v_n&\stackrel{\eqref{XIX-Pars1}}{=}&y_n-V_E^*\alpha_n\\
&\stackrel{\eqref{XVIII-Pars2}}{=}&y_n-V_E^*(U_EV_E^*-I)^{-1}U_Ey_n\\
&\stackrel{\eqref{medjukorak}}{=}&y_n-V_E^*U_E(V_E^*U_E-I)^{-1}y_n\\
&=&y_n-\left((V_E^*U_E-I)+I\right)(V_E^*U_E-I)^{-1}y_n\\
&=&y_n-y_n-(V_E^*U_E-I)^{-1}y_n\\
&=&(I-V_E^*U_E)^{-1}y_n.
\end{eqnarray*}
\end{proof}

\vspace{.1in}

Recall from Theorem~\ref{main2} that $v_n$'s can be obtained by solving a system of linear equations. The preceding theorem provides us with another possibility: $(v_n)_{n\in E}$ is identified as the image of $(y_n)_{n\in E}$
under the action of $(I-V_E^*U_E)^{-1}$. Put $E=\{n_1,n_2,\ldots n_k\}$, $k\in \Bbb N$. Then $V_E^*U_E$ is of the form $$V_E^*U_Ex=\sum_{i=1}^k\langle x,x_{n_i}\rangle y_{n_i}=\sum_{i=1}^k\theta_{y_{n_i},x_{n_i}}(x),\quad x\in H.$$
Thus, in applications we need an efficient procedure for computing $(I-\sum_{i=1}^k\theta_{y_{n_i},x_{n_i}})^{-1}$.

\vspace{.1in}

The case $k=1$ is easy. Here we need the following simple observation: if  $x,y\in H$ are such that $I-\theta_{y,x}$ is invertible, then  $\la y,x\ra\neq 1$. (Indeed, $\langle y,x\rangle=1$ would imply $(I-\theta_{y,x})y=y-\langle y,x\rangle y=0$ and, by invertibility of  $I-\theta_{y,x}$, we would have $y=0$ which contradicts the equality $\la y,x\ra= 1$.)

Moreover, it can be easily verified that $(I-\theta_{y,x})^{-1}$, if it exists, is given by
\begin{equation}\label{cijeli inverz}
\left(I-\theta_{y,x}\right)^{-1}=I+\frac{1}{1-\la y,x\ra}\theta_{y,x}.
\end{equation}

\begin{cor}\label{main2-cor2}
Let $(x_n)_n$ be a frame for a Hilbert space $H$ with the analysis operator $U$ and the canonical dual $(y_n)_n.$ Suppose that a set $E=\{m\}$ satisfies the minimal redundancy condition for $(x_n)_n$.
Let $v_m=0$ and
\begin{equation}\label{XIX-cor}
v_n= y_n +\frac{\la y_n,x_m\ra}{1-\la y_m,x_m\ra}y_{m},\quad \forall n\neq m.
\end{equation}
Then $(v_n)_n$ is a frame for $H$ dual to $(x_n)_n$.
\end{cor}
\begin{proof} If $E=\{m\}$ then $I-V_E^*U_E=I-\theta_{y_m,x_m}.$ Now \eqref{cijeli inverz} and Theorem~\ref{main2-cor1} give
\eqref{XIX-cor}.
\end{proof}

\vspace{.1in}

Suppose now that $E$ has $k\geq 2$ elements and satisfies the minimal redundancy condition for a frame $(x_n)_n$ with the analysis operators $U.$ Assume for notational simplicity that $E=\{1,2,\ldots, k\}$. As before we denote by $(y_n)_n$ the canonical dual of $(x_n)_n$ and by $V$ its analysis operator. Observe that $I-V_E^*U_E=I-\sum_{i=1}^k\theta_{y_i,x_i}$. Recall from Lemma~\ref{bivsa lema 2 u sec 4} that $I-V_F^*U_F$ is invertible for all $F\subseteq E$. In particular, by taking $F=\{1\}, F=\{1,2\}, F=\{1,2,3\}, \ldots$ we conclude that $I-\sum_{i=1}^n\theta_{y_i,x_i}$ is an invertible operator for all $n=1,2,\ldots,k$.

In the theorem that follows we demonstrate an iterative procedure for computing all $(I-\theta_{y_1,x_1})^{-1}$, $(I-\sum_{i=1}^2\theta_{y_i,x_i})^{-1}$, $(I-\sum_{i=1}^3\theta_{y_i,x_i})^{-1}$, $\ldots$  $(I-\sum_{i=1}^k\theta_{y_i,x_i})^{-1}$ such that each $(I-\sum_{i=1}^n\theta_{y_i,x_i})^{-1}$, $n=1,2, \ldots,k$, is expressed as a product of exactly $n$ simple inverses $(I-\theta_{y,x})^{-1}$ which one obtains using \eqref{cijeli inverz}.

\begin{theorem}\label{thm-inverse}
Let $(x_n)_n$ be a frame for a Hilbert space $H$ with the analysis operator $U$ and the canonical dual $(y_n)_n.$ Suppose that a set $E=\{1,2,\ldots,k\}$, $k\in \Bbb N$, satisfies the minimal redundancy condition for $(x_n)_n$. Let $\overline{y}_1,\ldots,\overline{y}_n$ be defined as
\begin{eqnarray}\label{definicija krnjih nizova}
\overline{y}_1&=&y_1, \nonumber\\
\overline{y}_n&=&(I-\theta_{\overline{y}_{n-1},x_{n-1}})^{-1}\ldots(I-\theta_{\overline{y}_1,x_1})^{-1}y_n,\quad n=2,\ldots,k.
\end{eqnarray}
Then $\overline{y}_1,\ldots,\overline{y}_n$ are well defined and
\begin{equation}\label{inverzi}
(I-\sum_{i=1}^n\theta_{y_i,x_i})^{-1}=(I-\theta_{\overline{y}_{n},x_{n}})^{-1}\dots(I-\theta_{\overline{y}_1,x_1})^{-1},\quad n=1,\ldots, k.
\end{equation}
\end{theorem}
\begin{proof}
In order to see that $\overline{y}_1,\ldots,\overline{y}_k$ are well defined we have to prove that the operators $I-\theta_{\overline{y}_n,x_n}$ for $n=1,\ldots,k$ are invertible.
Note that, as already observed, $I-\sum_{i=1}^n\theta_{y_i,x_i}$ are invertible for all $n=1,2,\ldots, k$.

We prove by induction.

For $n=1$ we have $I-\theta_{\overline{y}_1,x_1}=I-\theta_{y_1,x_1}$ which is an invertible operator by Lemma~\ref{bivsa lema 2 u sec 4} (applied to $F=\{1\}$), and formula \eqref{inverzi} is trivially satisfied.

Assume now that for some $n<k$ the operators  $I-\theta_{\overline{y}_1,x_1},\ldots,I-\theta_{\overline{y}_n,x_n}$ are invertible and that \eqref{inverzi} is satisfied.

Observe the equality
\begin{equation}\label{product rule}
T\theta_{y,x}=\theta_{Ty,x}
\end{equation}
which holds for all $x,y\in H$ and $T\in \Bbb B(H)$.
Now  we have
\begin{eqnarray*}\label{inverz-n}
I-\theta_{\overline{y}_{n+1},x_{n+1}}&\stackrel{\eqref{definicija krnjih nizova}}{=}&I-\theta_{(I-\theta_{x_n,\overline{y}_n})^{-1}\cdots(I-\theta_{x_1,\overline{y}_1})^{-1}y_{n+1},x_{n+1}}\\
&\stackrel{\eqref{product rule}}{=} &I-(I-\theta_{\overline{y}_n,x_n})^{-1}\cdots(I-\theta_{\overline{y}_1,x_1})^{-1}\theta_{y_{n+1},x_{n+1}}\\
&\stackrel{\eqref{inverzi} }{=}&I-(I-\sum_{i=1}^n\theta_{y_i,x_i})^{-1}\theta_{y_{n+1},x_{n+1}}\\
&=&(I-\sum_{i=1}^n\theta_{y_i,x_i})^{-1}(I-\sum_{i=1}^n\theta_{y_i,x_i})-(I-\sum_{i=1}^n\theta_{y_i,x_i})^{-1}\theta_{y_{n+1},x_{n+1}}\\
&=&(I-\sum_{i=1}^n\theta_{y_i,x_i})^{-1}(I-\sum_{i=1}^n\theta_{y_i,x_i}-\theta_{y_{n+1},x_{n+1}})\\
&=&(I-\sum_{i=1}^n\theta_{y_i,x_i})^{-1}(I-\sum_{i=1}^{n+1}\theta_{y_i,x_i})\\
&\stackrel{\eqref{inverzi}}{=}&(I-\theta_{\overline{y}_{n},x_{n}})^{-1}\cdots(I-\theta_{\overline{y}_1,x_1})^{-1}(I-\sum_{i=1}^{n+1}\theta_{y_i,x_i}).
\end{eqnarray*}
This proves that $I-\theta_{\overline{y}_{n+1},x_{n+1}}$ is  invertible (as a product of invertible operators).
Also, it follows from the final equality that
$$(I-\sum_{i=1}^{n+1}\theta_{y_i,x_i})^{-1}=
(I-\theta_{\overline{y}_{n+1},x_{n+1}})^{-1}(I-\theta_{\overline{y}_{n},x_{n}})^{-1}\cdots(I-\theta_{\overline{y}_1,x_1,})^{-1}.$$
\end{proof}

\vspace{0.1in}
We conclude this section with the result that improves Theorem~6.2 from \cite{LS} by removing the linear independence assumption. In this way we provide a closed-form formula for the inverse $(I-\sum_{n=1}^{k}\theta_{y_k,x_k})^{-1}$ which can (alternatively) be used, via Theorem~\ref{main2-cor1}, for obtaining our dual frame $(v_n)_n$.

\begin{theorem}\label{tm-LS}
Let $x_1,\ldots,x_k$ and $y_1,\ldots,y_k$ be vectors in a Hilbert space $H$ such that the operator $R=I-\sum_{j=1}^k\theta _{y_j,x_j}\in\Bbb B(H)$ is invertible.
Then
$$R^{-1}=I+\sum_{i,j=1}^kc_{ij}\theta_{y_i,x_j},$$
where the coefficient matrix $C:=(c_{ij})$ is given by
\begin{equation}\label{opet_minus}
C=-\left[
  \begin{array}{cccc}
    \la y_1,x_1\ra-1 & \la y_2,x_1\ra&\ldots & \la y_k,x_1\ra \\
    \la y_1,x_2\ra & \la y_2,x_2\ra-1&\ldots & \la y_k,x_2\ra \\
    \vdots & \vdots& & \vdots \\
     \la y_1,x_k\ra & \la y_2,x_k\ra&\ldots & \la y_k,x_k\ra-1 \\
 \end{array}
\right]^{-1}.
\end{equation}
\end{theorem}
\begin{proof}
Let $U,V:H\to \Bbb C^k$ be the analysis operators of Bessel sequences $(x_n)_{n=1}^k$, $(y_n)_{n=1}^k,$ respectively. Then $R=I-V^*U,$ and since $R$ is invertible, $UV^*-I$ is also an invertible operator.

Let $(e_n)_{n=1}^k$ be the canonical basis  for $\Bbb C^k$. Since
$$(UV^*-I)_{ij}=\la (UV^*-I)e_j,e_i\ra=\la y_j,x_i\ra-\delta_{ij},$$ ($\delta_{ij}$ is the Kronecker delta) the matrix representation of the operator $UV^*-I: \Bbb C^k\to \Bbb C^k$ with respect to the basis $(e_n)_{n=1}^k$ is precisely the matrix
$$\left[
  \begin{array}{cccc}
    \la y_1,x_1\ra-1 & \la y_2,x_1\ra&\ldots & \la y_k,x_1\ra \\
    \la y_1,x_2\ra & \la y_2,x_2\ra-1&\ldots & \la y_k,x_2\ra \\
    \vdots & \vdots& & \vdots \\
     \la y_1,x_k\ra & \la y_2,x_k\ra&\ldots & \la y_k,x_k\ra-1 \\
 \end{array}
\right].$$
As a matrix representation of an invertible operator, this is an invertible matrix.

Now one proceeds exactly as in the proof of Theorem~6.2 from \cite{DS}.
\end{proof}


\section{Uniformly redundant frames}

Recall from the introduction that a frame $(x_n)_n$ is said to be $M$-robust, $M\in \Bbb N$, if any set of indices of cardinality $M$ satisfies the minimal redundancy condition for $(x_n)_n$. Obviously, such frames are resistant to erasures of any $M$ frame coefficients. A special subclass consists of $M$-robust frames with the property that after removal of any $M$ elements the reduced sequence makes up a basis for the ambient space - we say that such frames are of uniform excess $M$. In particular, if a frame $(x_n)_{n=1}^{N+M}$ for an $N$-dimensional space $H$ is of uniform excess $M$, any $N$ of its members make up a basis for $H$. Such frames are sometimes called \emph{full spark frames} (see \cite{ACM}) or \emph{maximally robust frames} (as in \cite{PK}).

Here we provide some results on such frames. Let us start with a simple proposition that is certainly known. For completeness we include a proof.

\begin{prop}\label{bounded from below}
Let $(x_n)_n$ be a frame for a Hilbert space $H$ with the lower frame bound $A$. If $\|x_j\|<\sqrt{A}$ for some index $m$ then the set $E=\{m\}$ satisfies the minimal redundancy condition for $(x_n)_n$.
\end{prop}
\begin{proof}
Let $(y_n)_n$ be the canonical dual of $(x_n)_n$. Denote again by $U$ and $V$ the analysis operators of $(x_n)_n$ and $(y_n)_n$, respectively.
Since $\|(U^*U)^{-1}\|\le \frac{1}{A}$ it follows that
$$\la y_m,x_m\ra=\la (U^*U)^{-1}x_m,x_m\ra\le \| (U^*U)^{-1}\| \|x_m\|^2<1.$$ Then $I+\frac{1}{1-\la y_m,x_m\ra}\theta_{y_m,x_m}$ is well defined and, by \eqref{cijeli inverz}, it is the inverse of
$I-\theta_{y_j,x_j}=I-V_E^*U_E.$ By Lemma~\ref{bivsa lema 2 u sec 4}, $E$ satisfies the minimal redundancy condition for $(x_n)_n$.
\end{proof}

\begin{cor}\label{1-robust}
Suppose that $(x_n)_n$ is a frame for a Hilbert space $H$ with the lower frame bound $A$ such that $\|x_n\|<\sqrt{A}$ for all $n\in \Bbb N$. Then $(x_n)_n$ is $1$-robust.
\end{cor}

\vspace{.1in}

Consider now a Parseval frame $(x_n)_n$ for a Hilbert space $H$. By Proposition~\ref{bounded from below}, a sequence that is obtained from $(x_n)_n$  by removing any $x_n$ such that $\|x_n\|<1$ is again a frame for $H$. On the other hand, it is a well known fact each $x_m$ such that $\|x_m\|=1$ is orthogonal to all $x_n,\,n\not=m$. By combining these two facts we obtain

\begin{cor}\label{Parseval 1-robust}
Let $(x_n)_n$ be a Parseval frame for a Hilbert space $H$. The following two conditions are mutually equivalent:
\begin{enumerate}
  \item $(x_n)_n$ is $1$-robust;
  \item $\|x_n\|<1$ for all $n\in \Bbb N$.
  \end{enumerate}
\end{cor}

\vspace{.1in}

Note in passing that a similar statement can be proved for an arbitrary frame $(x_n)_n$ with the analysis operator $U$ using the preceding corollary applied to the associated Parseval frame $((U^*U)^{-\frac{1}{2}}x_n)_n$. We omit the details.

\vspace{.1in}

Next we show that each Parseval frame that is not an orthonormal basis can be converted to a $1$-robust Parseval frame. We first need a lemma which is proved using a nice maneuver from \cite{BP}.

\begin{lemma}\label{Paulsen}
Let $(x_n)_n$ be a Parseval frame for a Hilbert space $H$ that is not an orthonormal basis. Suppose that   $\|x_i\|=1$ for some $i\in \Bbb N$.
Then there exists $j\in \Bbb N$, $j\not =i$, such that $\|x_j\|<1$ and a Parseval frame $(x_n^{\prime})_n$ for $H$ with the properties $\|x_i^{\prime}\|<1$, $\|x_j^{\prime}\|<1$, and
$x_n^{\prime}=x_n$ for all $n\not=i,j$.
\end{lemma}
\begin{proof}
First, since $(x_n)_n$ is not an orthonormal basis, there exists at least one index $j\in \Bbb N$, $j\not =i$, such that $\|x_j\|<1$. Find such $j$, take a real number $\varphi$, $0<\varphi<\frac{\pi}{2}$,  and define
$$x_n^{\prime}=\left\{
        \begin{array}{cl}
          \cos\varphi\, x_i+\sin\varphi\, x_j, & \hbox{if $n=i$;} \\
          -\sin\varphi\, x_i+\cos\varphi\, x_j, & \hbox{if $n=j$;} \\
          x_n, & \hbox{if $n\not=i,j$.}
        \end{array}
      \right.$$
 A direct verification shows that $(x_n^{\prime})_n$ is a Parseval frame for $H$. Further, $\|x_i\|=1$ implies that $x_i\perp x_n$ for all $n\not =i$.
Using this, we find
$$\|x_i^{\prime}\|^2=\|x_i\|^2\cos^2 \varphi + \|x_j\|^2\sin^2\varphi<1$$
and
$$\|x_j^{\prime}\|^2= \|x_i\|^2\sin^2 \varphi+ \|x_j\|^2\cos^2\varphi<1.$$
\end{proof}

\vspace{.1in}

\begin{remark}\label{producing 1-robust}
Suppose we are given a Parseval frame $(x_n)_n$ for a Hilbert space $H$ that is not an orthonormal basis. If $\|x_n\|<1$ for all $n\in \Bbb N$, Corollary~\ref{Parseval 1-robust} guarantees that $(x_n)_n$ is $1$-robust.

If, on the other hand, there exists $i\in \Bbb N$ such that $\|x_i\|=1$ we can apply the preceding lemma to obtain another Parseval frame $(x_n^{\prime})_n$ with less elements that are perpendicular to all other members of the frame. By repeating the procedure we eventually obtain a $1$-robust Parseval frame.
\end{remark}

\vspace{.1in}

It is now natural to ask if there is a similar technique that would produce a $2$-robust Parseval frame starting from a $1$-robust Parseval frame.
A related question is: is there some necessary and sufficient condition (possibly stronger than $\|x_n\|<1$) on the norms of the frame members which would ensure $2$-robustness? The following example shows that the answer to the latter question is negative.

\begin{example} Let $\varepsilon\in\la 0,\frac{1}{2}\ra.$ Let $(\epsilon_n)_n$ be an orthonormal basis for a Hilbert space $H$.
Let
$$x_n=\left\{
        \begin{array}{ll}
          \sqrt{\varepsilon}\epsilon_{k}, & \hbox{if $n=3k-2$ or $n=3k-1$;} \\
          \sqrt{1-2\varepsilon}\epsilon_{k}, & \hbox{if $n=3k$;}
        \end{array}
      \right.$$
The sequence $(x_n)_n$ is obviously a $2$-robust Parseval frame. Observe that $\|x_{3k}\|= \sqrt{1-2\varepsilon}$ for all $k,$ so
we can choose $\epsilon$ such that the norms $\|x_{3k}\|$ are arbitrarily close to $1$.
\end{example}

\vspace{.1in}

In the rest of the paper we turn to finite-dimensional spaces and their finite frames. Our goal is to characterize full spark frames (i.e.~finite frames of uniform excess). Let us begin with two simple examples.

\begin{example}\label{ex_N+1_full_spark}
Suppose $(x_1,x_2,\ldots,x_N)$ is a basis for a Hilbert space $H$. Choose arbitrary $\lambda_1, \lambda_2,\ldots, \lambda_N$ such that $\lambda_i \not =0$ for all $i=1,2,\ldots,N$, and define
$$
x_{N+1}=\lambda_1 x_1+\lambda_2 x_2+\ldots + \lambda_N x_N.
$$
Then it is easy to verify that $(x_n)_{n=1}^{N+1}$ is a full spark frame for $H$.
\end{example}

\vspace{.1in}

\begin{example} \label{ex_N+2_full_spark}
Take again a basis $(x_1,x_2,\ldots,x_N)$ for $H.$ Choose arbitrary $\lambda_i,\mu_i, i=1,\ldots,N$ such that
$\lambda_i,\mu_i \neq 0$ for all $i=1,2,\ldots,N$, and $\frac{\lambda_i}{\mu_i}\not =\frac{\lambda_j}{\mu_j}$ for $i\not =j$. Let
\begin{eqnarray*}
x_{N+1}&=&\lambda_1x_1+\lambda_2x_2+\ldots+\lambda_Nx_N, \\
x_{N+2}&=&\mu_1x_1+\mu_2x_2+\ldots+\mu_Nx_N.
\end{eqnarray*}
Then $(x_n)_{n=1}^{N+2}$ is a full spark frame for $H$.

To see this, we must show that any subsequence consisting of exactly $N$ vectors makes up a basis for $H$. For example, consider  $(x_{N+1},x_{N+2},x_3,\ldots,x_N).$ The matrix representation of these $N$ vectors in the basis $(x_1,x_2,\ldots,x_N)$ is
$$A=\left[ \begin{array}{ccccc}
      \lambda_1 & \mu_1 & 0 & \ldots& 0 \\
      \lambda_2 & \mu_2 & 0 & \ldots& 0 \\
      \lambda_3 & \mu_3 & 1 & \ldots& 0 \\
      \vdots & \vdots & \vdots & & \vdots \\
      \lambda_N & \mu_N & 0 & \ldots& 1 \\
    \end{array}
  \right].$$
Since $\det A=\lambda_1\mu_2-\lambda_2\mu_1\neq 0,$ $A$ is invertible and $(x_{N+1},x_{N+2},x_3,\ldots,x_N)$ is a basis for $H.$
Again, we omit the further details.
\end{example}
\vspace{.2in}

The preceding two examples are just special cases of the following theorem.

\begin{theorem}\label{2N}
Let $(x_n)_{n=1}^N$, $N\in \Bbb N$, be a basis of a Hilbert space $H$ and let $T=(t_{ij})\in M_{NM}$, $M\in \Bbb N$, be a matrix with the property that each square submatrix of $T$ is invertible. Define $x_{N+1},x_{N+2},\ldots,x_{N+M}\in H$ by
\begin{equation}\label{full spark vectors}
x_{N+j}=\sum_{i=1}^Nt_{ij}x_i,\quad \forall j=1,2,\ldots,M.
\end{equation}
Then $(x_n)_{n=1}^{N+M}$ is a full spark frame for $H$.

Conversely, each full spark frame for $H$ is of this form. More precisely, if $(x_n)_{n=1}^{N+M}$ is a full spark frame for $H$ then there is a matrix $T=(t_{ij})\in M_{NM}$ such that
$x_{N+1},x_{N+2},\ldots,x_{N+M}$ are of the form \eqref{full spark vectors} and whose all square submatrices are invertible.
\end{theorem}
\begin{proof}
Suppose that we are given a basis $(x_n)_{n=1}^N$ for $H$ and a matrix $T=(t_{ij})\in M_{NM}$ whose all square submatrices are invertible. Consider $(x_n)_{n=1}^{N+M}$ with $x_{N+1},x_{N+2},\ldots,x_{N+M}\in H$ defined by \eqref{full spark vectors}.
Let $k$ be a natural number such that $1\leq k\leq N,M$. Consider two arbitrary sets of indices of cardinality $k$; $I=\{i_1,i_2,\ldots,i_k\}$ and
$J=\{j_1,j_2,\ldots,j_k\}$ and let $I^c=\{1,2,\ldots,N\}\setminus I$. We must prove that a reduced sequence
\begin{equation}\label{2Nfull spark proof}
(x_n)_{n\in I^c} \cup (x_{N+j})_{j\in J}
\end{equation}
is a basis for $H$.
(Note that the case $k=0$ is trivial.)

Let $C\in M_N$ denotes the matrix that is obtained by representing our reduced sequence \eqref{2Nfull spark proof} in the basis $(x_n)_{n=1}^N$. It suffices to show that $C$ is an invertible matrix. We shall show, using an argument from the proof of Theorem 6 in \cite{ACM} that $\det C\not =0$. By suitable changes of rows and columns of $C$, where only the first $N-k$ columns of $C$ are involved, we get a block-matrix $C^{\prime}$ of the form
$$
C^{\prime}=\left[\begin{array}{cc}I_{N-k}&T^{\prime}\\0&T^{\prime \prime}\end{array}\right]
$$
where $I_{N-k} \in M_{n-k}$ is a unit matrix while $T^{\prime}\in M_{N-k,k}$ and $T^{\prime \prime}\in M_k$ are some submatrices of $T$. By the hypothesis, $T^{\prime \prime}$ is invertible. Hence, $\det C^{\prime}=\det I_{N-k} \cdot \det T^{\prime \prime}=
\det T^{\prime \prime} \not =0$ and this obviously implies $\det C\not =0$.

To prove the converse, suppose that $(x_n)_{n=1}^{N+M}$ is an arbitrary full spark frame for $H$. In particular, $(x_n)_{n=1}^{N}$ is a basis for $H$, so there exist numbers $t_{ij}$ such that $x_{N+1},x_{N+2},\ldots,x_{N+M}$ are of the form \eqref{full spark vectors}. We must prove that each square submatrix of $T=(t_{ij})\in M_{NM}$ is invertible.

Consider two sets of indices $I=\{i_1,i_2,\ldots,i_k\}$ and
$J=\{j_1,j_2,\ldots,j_k\}$ with $1\leq k\leq N,M$ and the corresponding $k\times k$ submatrix $T_{I,J}=(t_{ij})_{i\in I,j\in J}$ of $T$.
Denote again $I^c=\{1,2,\ldots,N\}\setminus I$ and consider a reduced sequence
$$
(x_n)_{n\in I^c} \cup (x_{N+j})_{j\in J}.
$$
By the assumption, these $N$ vectors make up a basis for $H$. Denote by $C$ the matrix representation of this basis with respect to $(x_n)_{n=1}^N$ and notice that $C$ is an invertible matrix. In particular, the rows of $C$ are linearly independent. Observe now that a $k\times N$ submatrix $C_I$ of $C$ that corresponds to the rows indexed by $I$ is a block-matrix of the form
$$
C_I=\left[\begin{array}{ccc}0&|&T_{I,J}\end{array} \right].
$$
In particular, the rows of $C_I$ are linearly independent. This immediately implies that the rows of $T_{I,J}$ are linearly independent. Thus, $T_{I,J}$ is invertible.
\end{proof}

\vspace{.1in}

\begin{remark} Let $(x_n)_{n=1}^{N+M}$ be a full spark frame for an $N$-dimensional Hilbert space $H$. Consider a matrix $T$ introduced by
\eqref{full spark vectors} as in the Theorem~\ref{2N}.

(a) Let $M=1.$ Then
$T=\left[\begin{array}{c}
t_{11}\\t_{21}\\\vdots\\t_{N1}
\end{array}
\right].$
Obviously, the property that each square submatrix of $T$ is invertible means that $t_{i1}\neq 0$ for all $i=1,\ldots,N.$
Therefore, all full spark frames for $N$-dimensional Hilbert space $H$ consisting of $N+1$ elements are of the form as in Example~\ref{ex_N+1_full_spark}.

(b) Let $M=2.$ Then
$T=\left[\begin{array}{cc}
t_{11}&t_{12}\\t_{21}&t_{22}\\\vdots&\vdots\\t_{N1}&t_{12}
\end{array}
\right].$
Then each square submatrix of $T$ will be invertible if and only if $t_{ij}\neq 0,i=1,\ldots,N,j=1,2,$ and $t_{i1}t_{j2}\neq t_{i2}t_{j1},$ i.e. $\frac{t_{i1}}{t_{j1}}\neq \frac{t_{i2}}{t_{j2}},$  for all $i,j=1,\ldots,N,i\neq j.$ In other words, every full spark frames for $N$-dimensional Hilbert space $H$ with  $N+2$ elements is as in Example~\ref{ex_N+2_full_spark}.

(c) Let $M=N.$ Recall that a matrix $T\in M_N$ is said to be totally nonsingular if each minor of $T$ is different from zero (i.e.~if each $k \times k$ submatrix of $T$, for all $k=1,2,\ldots, N$, is nonsingular). Therefore, a sequence $(x_n)_{n=1}^{2N}$ is a full spark frame for $H$ if and only if $T$ is totally nonsingular.
\end{remark}

\vspace{.1in}

In the rest of the paper we discuss some applications of Theorem~\ref{2N}. Let us begin by  providing examples of totally nonsingular matrices.

Recall from \cite{FZ} that a square matrix $T$ is called {\em totally positive} if all its minors are positive real numbers. Clearly, each totally positive matrix is totally nonsingular. We shall construct infinite totally positive symmetric matrices which can be used, via Theorem~\ref{2N}, for producing new examples of full spark frames. A construction that follows may be of its own interest.

For a matrix $T=(t_{ij})\in M_n$ and two sets of indices $I,J\subseteq \{1,2,\ldots, n\}$ of the same cardinality we denote by $\Delta(T)_{I,J}$ the corresponding minor; i.e.~ the determinant of a submatrix $T_{I,J}=(t_{ij})_{i\in I,j\in J}$. A minor $\Delta(T)_{I,J}$ is called {\em solid} if both $I$ and $J$ consist of consecutive indices. More specifically, a minor $\Delta(T)_{I,J}$ is called {\em initial} if it is solid and $1\in I\cup J$.
Observe that each matrix entry is the lower-right corner of exactly one initial minor. In our construction we will make use of the following efficient criterion for total positivity which was proved by M. Gasca and J.M. Pe\~{n}a in \cite{GP} (see also Theorem 9 in \cite{FZ}): a square matrix is totally positive if and only if all its initial minors are positive.

\vspace{.1in}

To describe our construction we need to introduce one more notational convention. Given an infinite matrix $T=(t_{ij})_{i,j=1}^{\infty}$ and $n\in \Bbb N$, we denote by $T^{(n)}$ a submatrix in the upper-left $n \times n$ corner of $T$, that is $T^{(n)}=(t_{ij})_{i,j=1}^{n}$. Its minors will be denoted by $\Delta(T^{(n)})_{I,J}$.

\begin{theorem}\label{constructing TP}
Let $(a_n)_n$ and $(b_n)_n$ be sequences of natural numbers such that $b_1=a_2$ and $a_nb_{n+1}-b_na_{n+1}=1$ for all $n \in \Bbb N$. There exists an infinite matrix $T=(t_{ij})_{i,j=1}^{\infty}$ with the following properties:
\begin{enumerate}
\item $t_{ij}\in \Bbb N,\,\forall i,j \in \Bbb N$;
\item $t_{ij}=t_{ji},\,\forall i,j\in \Bbb N$;
\item $t_{1n}=t_{n1}=a_n,\,\forall n\in \Bbb N$, and $t_{2n}=t_{n2}=b_n,\,\forall n\in \Bbb N$.
\item all minors of $T^{(n)}$ are positive (i.e. $T^{(n)}$ is totally positive), for each $n\in \Bbb N$ ;
\item For each $n \in \Bbb N$, it holds

\noindent
$\Delta(T^{(n)})_{\{n\},\{1\}}=t_{n1}=a_n$,

\noindent
$\Delta(T^{(n)})_{\{n,n-1\},\{1,2\}}=1$,

\noindent
$\Delta(T^{(n)})_{\{n,n-1,n-2\},\{1,2,3\}}=1$,

\noindent
$\Delta(T^{(N)})_{\{n,n-1,n-2,n-3\},\{1,2,3,4\}}=1$,

\noindent
$\ldots$

\noindent
$\Delta(T^{(n)})_{\{n,n-1,\ldots,1\},\{1,2,\ldots,n\}}=\det T^{(n)}=1$

\noindent
(i.e. all solid minors of $(T^{(n)})$ with the lower-left corner coinciding with the lower-left corner of $(T^{(n)})$, except possibly $\Delta(T^{(n)})_{\{n\},\{1\}}$, are equal to $1$).
\end{enumerate}
\end{theorem}
\begin{proof}
We shall construct $T$ by induction starting from $T^{(1)}=\left[ a_1 \right]$. Observe that $T^{(2)}=\left[\begin{array}{cc}a_1&b_1\\a_2&b_2\end{array}\right]$; note that $T^{(2)}$ is symmetric since by assumption we have $b_1=a_2$.

\vspace{.1in}

Suppose that we have a symmetric totally positive matrix with integer coefficients $T^{(n)}\in M_n$ which satisfies the above conditions (1)-(5),
$$
T^{(n)}=\left[\begin{array}{ccccc}
a_1&a_2&a_3&\ldots&a_n\\
a_2&b_2&b_3&\ldots&b_n\\
a_3&b_3&t_{ž33}&\ldots&t_{3n}\\
\vdots&\vdots&\vdots& &\vdots\\
a_{n-1}&b_{n-1}&t_{n-1,3}&\ldots&t_{n-1,n}\\
a_n&b_n&t_{n3}&\ldots&t_{nn}
\end{array}\right].
$$
Put
\begin{equation}\label{Te en plus jedan}
T^{(n+1)}=\left[\begin{array}{cccccc}
a_1&a_2&a_3&\ldots&a_n&a_{n+1}\\
a_2&b_2&b_3&\ldots&b_n&b_{n+1}\\
a_3&b_3&t_{33}&\ldots&t_{3n}&x_3\\
\vdots&\vdots&\vdots& &\vdots&\vdots\\
a_{n-1}&b_{n-1}&t_{n-1,3}&\ldots&t_{n-1,n}&x_{n-1}\\
a_n&b_n&t_{n3}&\ldots&t_{nn}&x_n\\
a_{n+1}&b_{n+1}&x_3&\ldots&x_n&x_{n+1}
\end{array}\right].
\end{equation}
Note that, by the hypothesis on sequences $(a_n)_n$ and $(b_n)_n$, we have
$$
\det
\left[\begin{array}{cc}
a_n&b_n\\
a_{n+1}&b_{n+1}
\end{array}
\right]=1.
$$
We must find numbers $x_3,x_4,\ldots,x_n,y$ such that $T^{(n+1)}$ satisfies (1)-(5).
Consider a $3 \times 3$ minor in the lower-left corner of $T^{(n+1)}$:
$$
\Delta(T^{(n+1)})_{\{n+1,n,n-1\},\{1,2,3\}}=\det
\left[\begin{array}{ccc}
a_{n-1}&b_{n-1}&t_{n-1,3}\\
a_n&b_n&t_{n3}\\
a_{n+1}&b_{n+1}&x_3
\end{array}
\right].
$$
We can compute $\Delta(T^{(n+1)})_{\{n+1,n,n-1\},\{1,2,3\}}$ by the Laplace expansion along the third row. By the assumption on sequences $(a_n)_n$ and $(b_n)_n$ we have $\det \left[\begin{array}{cc}a_{n-1}&b_{n-1}\\a_n&b_n\end{array}\right]=1$; hence, there exists a unique integer $x_3$ such that
$\Delta(T^{(n+1)})_{\{n+1,n,n-1\},\{1,2,3\}}=1;$ choose this $x_3$ and put $t_{n+1,3}=x_3.$

Consider now
$$
\Delta(T^{(n+1)})_{\{n+1,n,n-1,n-2\},\{1,2,3,4\}}=\det \,
\left[\begin{array}{cccc}
a_{n-2}&b_{n-2}&t_{n-2,3}&t_{n-2,4}\\
a_{n-1}&b_{n-1}&t_{n-1,3}&t_{n-1,4}\\
a_n&b_n&t_{n3}&t_{n4}\\
a_{n+1}&b_{n+1}&t_{n+1,3}&x_4
\end{array}
\right].
$$
Note that the only unknown entry in this minor is $x_4$.
We again use the Laplace expansion along the bottom row. By the induction hypothesis we know that
$$
\det \,
\left[\begin{array}{ccc}
a_{n-1}&b_{n-2}&t_{n-2,3}\\a_{n-1}&b_{n-1}&t_{n-1,3}\\a_n&b_n&t_{n3}
\end{array}\right]=
\Delta(T^{(n)})_{\{n,n-1,n-2\},\{1,2,3\}}=1;
$$
hence, there is a unique $x_4\in \Bbb Z$ such that $\Delta(T^{(n+1)})_{\{n+1,n,n-1,n-2\},\{1,2,3,4\}}=1$. Put $t_{n+1,4}=x_4$.
We proceed in the same fashion to obtain $x_5,\ldots,x_{n+1}$ in order to achieve the above condition (5) for $T^{(n+1)}$.
Since $T^{(n+1)}$ is symmetric, all its essential minors with the lower-right corner in the last column are also equal to $1$. By the induction hypothesis $T^{(n)}$ is totally positive, so all essential minors of $T^{(n+1)}$ with the lower-right corner in the $i$th row and $j$th column such that $i,j\leq n$ are also positive. Thus, we can apply the above mentioned result of M. Gasca and J.M. Pe\~{n}a (Theorem 9 in \cite{FZ}) to conclude that $T^{(n+1)}$ is totally positive. In particular, the integers $x_3,x_4,\ldots,x_n,y$ that we have computed along the way are all positive. This concludes the induction step.
\end{proof}

\vspace{.1in}

\begin{example}\label{TPFibonacci}
Let us take $a_n=1$ and $b_n=n$, for all $n\in \Bbb N$. Clearly, the sequences $(a_n)_n$ and $(b_n)_n$ defined in this way satisfy the conditions from Theorem~\ref{constructing TP}. Thus, an application of that theorem gives us a totally positive matrix
$$
T=\left[\begin{array}{rrrrrrr}
1&1&1&1&1&1&\ldots\\
1&2&3&4&5&6&\ldots\\
1&3&6&10&15&21&\\
1&4&10&20&35&56&\\
1&5&15&35&70&126&\\
1&6&21&56&126&252&\\
\vdots&\vdots& & & &
\end{array}
\right]
$$

\vspace{.1in}

By the construction, the coefficients of $T$ in the first two rows and columns are determined in advance. One can prove that all other coefficients of $T$, those that need to be computed by an inductive procedure as described in the preceding proof, are given by
\begin{equation}\label{matrix fibonacci}
t_{i,j+1}=t_{ij}+t_{i-1,j+1},\quad \forall i\geq 3,\,\forall j\geq 2.
\end{equation}
This means that $T$ is in fact a well known Pascal matrix; i.e. that $t_{ij}$'s are given by $t_{ij}=\binom{i+j-2}{j-1}$ for all $i,j \geq 1$.
A verification of (\ref{matrix fibonacci}) serves as an alternative proof of Proposition~\ref{constructing TP} with this special choice of $(a_n)_n$ and $(b_n)_n$. The key observation is the equality that one obtains by subtracting each row in $\Delta(T^{(n+1)})_{\{n+1,n,\ldots,n-j+1\},\{1,2,\ldots,j+1\}}$ from the next one and then using \eqref{matrix fibonacci}:

$$
\Delta(T^{(n+1)})_{\{n+1,n,\ldots,n-j+1\},\{1,2,\ldots,j+1\}}=\left|\begin{array}{cc}
1&t_{n-j+1,2}\quad t_{n-j+1,3}\quad \ldots \quad t_{n-j+1,j+1}\\
\begin{array}{c}0\\0\\\vdots\\0\end{array}&\Delta(T^{(n+1)})_{\{n+1,n,\ldots,n-j+2\},\{1,2,\ldots,j\}}
\end{array}\right|.
$$
We omit the details.
\end{example}
\vspace{.1in}

Another example of an infinite totally positive symmetric matrix is obtained by a different choice of sequences $(a_n)_n$ and $(b_n)_n$.

\begin{example}\label{TP example 2}
Let $a_n=n$ and $b_n=3n-1$, for all $n\in \Bbb N$. Evidently, these two sequences satisfy the required conditions; namely, $b_1=a_2$ and $a_nb_{n+1}-b_na_{n+1}=1$ for all $n\in \Bbb N$. An application of Theorem~\ref{constructing TP} gives us a totally positive matrix
$$
T=\left[
\begin{array}{rrrrrrrrr}
1&2&3&4&5&6&7&8&\ldots\\
2&5&8&11&14&17&20&23&\ldots\\
3&8&14&21&29&38&48&59&\ldots\\
4&11&21&35&54&79&\cdot&\cdot&\\
5&14&29&54&94&\cdot&\cdot&\cdot&\\
6&17&38&79&\cdot&\cdot&\cdot&\cdot&\\
7&20&48&\cdot&\cdot&\cdot&\cdot&\cdot&\\
8&23&59&\cdot&\cdot&\cdot&\cdot&\cdot&\\
\vdots&\vdots&\vdots& & & & & &
\end{array}
\right]
$$
\end{example}

\vspace{.1in}

\begin{remark}\label{generating all TP}
Obviously, by choosing suitable sequences $(a_n)_n$ and $(b_n)_n$ one can generate in the same fashion  many other totaly positive symmetric matrices with integer coefficients.

It is also clear from the proof of Proposition~\ref{constructing TP} that, by applying a similar inductive procedure, one can construct any infinite totally positive matrix (not necessarily symmetric) with coefficients merely in $\Bbb R^+$, with a prescribed first column or the first row.
\end{remark}

\vspace{.1in}

Suppose now again that we work in an $N$-dimensional Hilbert space $H$. Recall that Theorem 6 from \cite{ACM} describes a procedure for producing full spark unit norm tight frames using Chebotar\"{e}v theorem on $N\times N$ discrete Fourier transform matrix with $N$ prime. Another technique for producing full spark equal norm tight frames can be found in \cite{PK}.

We can now provide, using Theorem~\ref{2N} and the preceding two examples, further examples of full spark frames for $H$ of arbitrary length. To do that, we only need to fix some $M\in\Bbb N$, and choose arbitrary set of indices $I=\{i_1,i_2,\ldots,i_N\}$, $J=\{j_1,j_2,\ldots, j_M\}$. Then we can take a totally positive matrix $T$ from Example~\ref{TPFibonacci} or Example~\ref{TP example 2}, its submatrix $T_{I,J}\in M_{NM}$ and apply Theorem~\ref{2N}. In this way we obtain a full spark frame for $H$ consisting of $N+M$ elements.

\vspace{.1in}

\begin{example}\label{novi genericki Fib}
Denote by $(x_n)_{n=1}^N$ the canonical basis for $\Bbb C^N$, $N\in \Bbb N$. Take arbitrary $M\in \Bbb N$ and the upper left $N\times M$ corner of the matrix from Example \ref{TPFibonacci}. An application of Theorem~\ref{2N} gives us a full spark frame $(x_n)_{n=1}^{N+M}$ for $\Bbb C^N$ whose members are represented in the basis $(x_n)_{n=1}^N$ by the matrix
$$
F_{NM}=\left[\begin{array}{cccccccccc}
1&0&0&\cdots&0&1&1&1&\cdots&1\\
0&1&0&\cdots&0&1&2&3&\cdots&M\\
0&0&1&\cdots&0&1&3&t_{33}&\cdots&t_{3M}\\
0&0&0&\cdots&0&1&4&t_{43}&\cdots&t_{4M}\\
\vdots&\vdots&\vdots& &\vdots&\vdots&\vdots&\vdots& &\vdots\\
0&0&0&\cdots&1&1&N&t_{N3}&\cdots&t_{NM}
\end{array}
\right]
$$
with
$$
t_{ij}=\binom{i+j-2}{j-1},\quad i=1,2,\ldots,N,\ j=1,2,\ldots,M.
$$
\end{example}

\vspace{.1in}

It would be useful to construct  classes of full spark frames with some additional properties. Given a frame $(x_n)_n$ for a Hilbert space $H$ with the analysis operator $U$, it is well known that $((U^*U)^{-\frac{1}{2}}x_n)_n$ is a Parseval frame for $H$. Since  $(U^*U)^{-\frac{1}{2}}$ is an invertible operator, this transformation preserves full spark property.

Unfortunately, computing the inverse of the positive  square root (from the frame operator) is costly. However, a class of full spark frames described as in Theorem~\ref{2N} in terms of an orthonormal basis is easier to handle.

Observe (as in the above example) that, if $(x_n)_{n=1}^{N+M}$ is a full spark frame and if
$T$ is a matrix introduced by
\eqref{full spark vectors} as in the Theorem~\ref{2N}, then the matrix representation of $(x_n)_{n=1}^{N+M}$ with respect to a basis $(x_n)_{n=1}^N$
is (written in the form of a block matrix) $\left[ \begin{array}{ccc}I&|&T\end{array} \right]$. Moreover, this is precisely the matrix of the synthesis operator $U^*$ of our frame $(x_n)_{n=1}^{N+M}$. If, additionally, $(x_n)_{n=1}^N$ is an orthonormal (e.g.~canonical) basis, this implies that $U^*U=I+TT^*$. An easy computation (which we omit) shows that, in this situation, $TT^*$ is nothing else than $\sum_{n=1}^M\theta_{x_{N+n},x_{N+n}}$. Thus,
\begin{equation}\label{specijalni oblik}
U^*U=I+\sum_{n=1}^M\theta_{x_{N+n},x_{N+n}}.
\end{equation}

\vspace{0.1in}

In the case $M=1$ there is a simple formula for $(I+\theta_{x,x})^{-\frac{1}{2}}$. The reader can easily check that
\begin{equation}\label{inverz}
\left(I+\theta_{x,x}\right)^{-\frac{1}{2}}=I+\frac{1}{\|x\|^2}\left(\frac{1}{\sqrt{1+\|x\|^2}}-1\right) \theta_{x,x},\quad \forall x\in H.
\end{equation}

\vspace{.1in}

For $M>1$ we do not have a formula for the inverse square root of $U^*U$ where
$U : \Bbb C^N \rightarrow \Bbb C^M$ is such that $U^*U$ is of the form \eqref{specijalni oblik}. However,
it suffices for our purposes to find any invertible operator $R$ on $\Bbb C^N$, not necessarily positive, such that $UR^*$ is an isometry. Then $(Rx_n)_n$ will be a Parseval full spark frame for $H$. Indeed, this is immediate from the observation that the analysis operator of $(Rx_n)_n$ is precisely $UR^*$.

In the theorem that follows we give a finite iterative procedure for constructing such an operator $R$. It turns out that $R$ is a finite product of operators of the form \eqref{inverz}.

\begin{theorem}\label{thm-PArs_spark new}
Let $U : \Bbb C^N \rightarrow \Bbb C^M$ be an operator such that $U^*U=I+\sum_{n=1}^M\theta_{f_n,f_n}$ for some $f_1,f_2,\ldots,f_M\in \Bbb C^N$, $M\in \Bbb N$. Let
$$f_n^{(0)}=f_n,\quad \forall n=1,\ldots, M, $$
and, for $k=1,\ldots, M$,
\begin{equation}\label{rekurzija_Pars}
f_n^{(k)}=\left(I+\theta_{f_{k}^{(k-1)},f_{k}^{(k-1)}}\right)^{-\frac{1}{2}}f_n^{(k-1)},\quad \forall n=1,\ldots, M.
\end{equation}
Consider the operators
\begin{equation}\label{T_Nk_Mk}
R_{k}^{(k-1)}=\left(I+\theta_{f_{k}^{(k-1)},f_{k}^{(k-1)}}\right)^{-\frac{1}{2}},\quad \forall k=1,\ldots,M,
\end{equation}
and
\begin{equation}\label{operator_T}
R= R_{M}^{(M-1)}R_{M-1}^{(M-2)}\cdots R_{1}^{(0)}.
\end{equation}
Then $R$ is an invertible operator such that $UR^*$ is an isometry.
\end{theorem}
\begin{proof}
Observe first that the recursion formulae \eqref{rekurzija_Pars} can be written as
\begin{equation}\label{rekurzija_Pars_1}
f_n^{(k)}=R_{k}^{(k-1)}f_n^{(k-1)},\quad \forall n=1,\ldots,M,\quad \forall k=1,\ldots, M.
\end{equation}
Notice  that we can rewrite the assumed equality $U^*U=I+\sum_{n=1}^M\theta_{f_n,f_n}$ can be  written as
$$U^*U\stackrel{\eqref{T_Nk_Mk} }{=}(R_{1}^{(0)})^{-2}+\theta_{f_{2},f_{2}}+\ldots+\theta_{f_{M},f_{M}}.$$
Multiplying on the both sides by $R_{1}^{(0)}$ and taking into account that
$$R_{1}^{(0)}\theta_{f_{k},f_{k}}R_{1}^{(0)}=\theta_{R_{1}^{(0)}f_{k},R_{1}^{(0)}f_{k}}\stackrel{\eqref{rekurzija_Pars_1} }{=}\theta_{f_{k}^{(1)},f_{k}^{(1)}}$$
for all $k=2,\ldots,M$, we get
$$R_{1}^{(0)}U^*UR_{1}^{(0)}=I+\theta_{f_{2}^{(1)},f_{2}^{(1)}}+\ldots+\theta_{f_{M}^{(1)},f_{M}^{(1)}}.$$
Again, we write this equality as
$$R_{1}^{(0)}U^*UR_{1}^{(0)}=(R_{2}^{(1)})^{-2}+\theta_{f_{3}^{(1)},f_{3}^{(1)}}+\ldots+\theta_{f_{M}^{(1)},f_{M}^{(1)}}.$$
Now we multiply on the both sides by $R_{2}^{(1)}$ and, as above, we get
$$R_{2}^{(1)}R_{1}^{(0)}U^*UR_{1}^{(0)}R_{2}^{(1)} =I+\theta_{f_{3}^{(2)},f_{3}^{(2)}}+\ldots+\theta_{f_{M}^{(2)},f_{M}^{(2)}}.$$
After $m$ steps we get
$$R_{M}^{(M-1)}\cdots R_{2}^{(1)}R_{1}^{(0)}U^*UR_{1}^{(0)}R_{2}^{(1)}\cdots R_{M}^{(M-1)} =I,$$
that is, by \eqref{operator_T},
$$
RU^*UR^*= I.
$$\end{proof}

As an immediate consequence  we get:

\begin{cor}
\label{thm-PArs_spark old}
Let $(x_n)_{n=1}^{N+M}$ be a frame for a Hilbert space $H$ such that $(x_n)_{n=1}^{N}$ is an orthonormal basis for $H.$
Let
\begin{equation}\label{rekurzija_Pars old1}
x_n^{(0)}=x_n,\quad \forall n=1,\ldots, N+M,
\end{equation}
and for $k=1,\ldots, M$
\begin{equation}\label{rekurzija_Pars old2}
x_n^{(k)}=\left(I+\theta_{x_{N+k}^{(k-1)},x_{N+k}^{(k-1)}}\right)^{-\frac{1}{2}}x_n^{(k-1)},\quad \forall n=1,\ldots, N+M.
\end{equation}
The sequence $(x_n^{(M)})_{n=1}^{N+M}$ is a Parseval frame for $H.$ If $(x_n)_{n=1}^{N+M}$ is full spark, then $(x_n^{(M)})_{n=1}^{N+M}$ is also a full spark frame.
\end{cor}
\begin{proof}
Denote by $U$ the analysis operator of $(x_n)_{n=1}^{N+M}$. As we already observed in the discussion preceding Theorem~\ref{thm-PArs_spark new}, the fact that $(x_n)_{n=1}^{N}$ is an orthonormal basis implies that the frame operator $U^*U$ is given by $U^*U=I+\sum_{n=1}^M\theta_{x_{N+n},x_{N+n}}$. We now apply the preceding theorem with $f_n=x_{N+n},\,n=1,2,\ldots,M$.

Having defined $x_n^{(k)}$, $k=0,1,\ldots, M$,  $n=1,2\ldots, N+M$, by \eqref{rekurzija_Pars old1} and \eqref{rekurzija_Pars old2} we can write (as in the preceding proof)
\begin{equation}\label{rekurzija_Pars old3}
x_n^{(k)}=R_k^{(k-1)}x_n^{(k-1)},\,\,\forall k=1,2,\ldots,M,\quad\forall n=1,2,\ldots,N+M,
\end{equation}
where
\begin{equation}\label{T_Nk_Mk old}
R_{k}^{(k-1)}=\left(I+\theta_{x_{N+k}^{(k-1)},x_{N+k}^{(k-1)}}\right)^{-\frac{1}{2}},\quad \forall k=1,\ldots,M.
\end{equation}
Put again
$$
R= R_{M}^{(M-1)}R_{M-1}^{(M-2)}\cdots R_{1}^{(0)}.
$$
Then \eqref{rekurzija_Pars old3} can be rewritten as
$$
x_n^{(M)}=R_M^{(M-1)}x_n^{(M-1)}=\ldots=R_M^{(M-1)}R_{M-1}^{(M-2)} \ldots R_1^{(0)}x_n^{(0)}
$$
i.e.
$$
x_n^{(M)}=Rx_n,\,\,\forall n=1,2,\ldots ,N+M.
$$
Since $R$ is surjective, $(x_n^{(M)})_{n=1}^{N+M}$ is a frame for $H$. Since its analysis operator $UR^*$ is an isometry, $(x_n^{(M)})_{n=1}^{N+M}$ is Parseval. Finally, since $R$ is invertible, if $(x_n)_{n=1}^{N+M}$  is a full spark frame,
$(x_n^{(M)})_{n=1}^{N+M}$ is also full spark.
\end{proof}

\vspace{.1in}

Here we note that the idea of transforming a general finite frame into a Parseval frame by a finite iterative process is not new. A different approach, inspired by the Gram-Schmidt orthogonalization procedure, can be found in \cite{CKut}.

\vspace{.1in}

\begin{example}
Consider a full spark frame $(x_n)_{n=1}^5$ for $\Bbb C^3$ from Example \ref{novi genericki Fib} with $N=3$ and $M=2$.
Recall that the matrix of its synthesis operator $U^*$ in the canonical pair of bases is
$$
U^*=F_{23}=\left[\begin{array}{ccccc}
1&0&0&1&1\\
0&1&0&1&2\\
0&0&1&1&3
\end{array}
\right]
$$

After applying the procedure described in the preceding corollary, we get the following sequence
\begin{eqnarray*}
x_1^{(2)}&=&\frac{5}{6}x_1+\frac{-4-\sqrt{6}}{60}x_2+\frac{2-2\sqrt{6}}{60}x_3\\
x_2^{(2)}&=&-\frac{1}{6}x_1+\frac{44+\sqrt{6}}{60}x_2+\frac{-22+2\sqrt{6}}{60}x_3\\
x_3^{(2)}&=&-\frac{1}{6}x_1+\frac{-28+3\sqrt{6}}{60}x_2+\frac{14+6\sqrt{6}}{60}x_3\\
x_4^{(2)}&=&\frac{1}{2}x_1+\frac{4+\sqrt{6}}{20}x_2+\frac{-1+\sqrt{6}}{10}x_3\\
x_5^{(2)}&=&\frac{\sqrt{6}}{6}x_2+\frac{2\sqrt{6}}{6}x_3,
\end{eqnarray*}
which makes up a full spark Parseval frame for $\Bbb C^3.$
\end{example}

\vspace{.1in}

\section*{Concluding remarks}

We have constructed, for every finite set of indices $\{n_1,n_2,\ldots,n_k\}$ that satisfies the minimal redundancy condition for a frame $(x_n)_n$, a dual frame $(v_n)_n$ which satisfies $v_n=0$ for all $n=n_1,n_2,\ldots,n_k$.
This dual frame enables us to reconstruct each signal $x$, using the reconstruction formula $x=\sum_{n=1}^{\infty}\la x,x_n\ra v_n$, without recovering (possibly) corrupted coefficients $\la x,x_{n_1}\ra$, $\la x,x_{n_2}\ra,
\ldots,\la x,x_{n_k}\ra$.

The elements of the dual frame $(v_n)_n$ are computed in terms of the canonical dual $(y_n)_n$. In Theorem \ref{main2} each $v_n$, $n\not =n_1,n_2,\ldots, n_k$,
is obtained in the form
$v_n=y_n-\sum_{i=1}^k\alpha_{ni}y_{n_i}$, where the $k$-tuple $(\alpha_{n1}, \alpha_{n2},\ldots,\alpha_{nk})$ is a unique solution of a system of $k$ linear equations in $k$ unknowns, whose matrix is independent of $n$.

In Theorem~\ref{main2-cor1} we have proved that $v_n=(I-\sum_{i=1}^k\theta_{y_{n_i},x_{n_i}})^{-1}y_n$, $n\not =n_1,n_2,\ldots, n_k$.

Related results, Theorem \ref{thm-inverse} and Theorem \ref{tm-LS} may be of independent interest. In Theorem \ref{thm-inverse} we demonstrated a finite iterative procedure for computing the inverse $(I-\sum_{i=1}^k\theta_{y_{n_i},x_{n_i}})^{-1}$. Theorem \ref{tm-LS}, improving a result form \cite{LS}, provides a closed-form formula for the same inverse operator.

Here we also mention that our dual frame $(v_n)_n$ (as any other dual frame) arises from an oblique projection $F$ to the range of the analysis operator of the original frame $(x_n)_n$. One can derive an explicit formula for $F$ in which the above coefficients $\alpha_{ni}$'s again play a role. One can also prove that there is a related oblique projection $Q$ whose infinite matrix with respect to the canonical basis of $\ell^2$ transforms $(x_n)_n$ into $(v_n)_n$ in the sense of Theorem 4 from \cite{A}. These two results and their proofs might be of some interest, but are omitted since they are not essential for the presentation. The details will appear elsewhere.

We have also discussed properties of frames robust to erasures. In Theorem \ref{2N} we have proved that all full spark frames for finite-dimensional Hilbert spaces are generated by matrices in which all square submatrices are nonsingular. In this light, Theorem \ref{constructing TP} serves as a useful tool for generating infinite totally positive matrices (those that have all minors strictly positive) and, consequently, for generating full spark frames.

Finally, we have obtained in Theorem \ref{thm-PArs_spark new} and Corollary \ref{thm-PArs_spark old} a finite iterative procedure that can be used for transforming general frames to Parseval ones. Remarkably, full spark property is preserved under that transformation. It would be interesting and useful for applications to refine this procedure in such a way that when applied (at least to some specific classes of frames) gives rise to Parseval full spark frames with some additional properties (e.g.~frames with all elements having equal norms).


\vspace{.1in}


\begin{thebibliography}{999}
\bibitem{A} A. Aldroubi, \emph{Portraits of frames}, Proc. AMS 123(1995), no. 6, 1661--1668.

\bibitem{ACM} B. Alexeev, J. Cahill, D. Mixon, \emph{Full spark frames}, J. Fourier Anal. Appl. 18 (2012), no. 6, 1167--1194.

\bibitem{BB} D. Baki\' c, T. Beri\' c, \emph{On excesses of frames}, Glasnik matemati\v cki, 50(2) 2015, 415--427.

\bibitem{BB2} D. Baki\' c, T. Beri\' c, \emph{Finite extensions of Bessel sequences}, Banach J. Math. Anal. 9(4) (2015), 1--13.

\bibitem{BCHL} R.~Balan, P.G.~Casazza, C.~Heil, Z.~Landau, \emph{Deficits nad excesses of frames}, Advances in Comp. Math.,
Special Issue on Frames, 18 (2003), 93-116.

\bibitem{BO} P. Boufounos, A.V. Oppenheim, \emph{Compensation of coefficients erasures in frame representations},
Proc. IEEE Int. Conf. Acoustics, Speech, and Signal Processing (ICASSP), May 14-19 2006.

\bibitem{BP} B.G. Bodman, V.I. Paulsen, \emph{Frames, graphs and erasures}, Lin. Alg. Appl., 404 (2005), 118--146.

\bibitem{CK} P.G. Casazza, J. Kova\v cevi\' c, \emph{Equal-norm tight frames with erasures}, Adv. Comp. Math., 18 (2003), 387--430.

\bibitem{CKu} P.G. Casazza, J. Kova\v cevi\' c, \emph{Finite frames}, in \emph{Applied and numerical harmonic analysis} (Theory and Applications), Springer, 2013.

\bibitem{CKut} P.G. Casazza, G. Kutyniok, \emph{A generalization of Gram-Schmidt orthogonalization generating all Parseval frames},  Adv. Comp. Math., 27 (2007), 65--78.

\bibitem{Ch} O. Christensen, \emph{An introduction to frames and Riesz bases}, Birkh\"{a}ser, 2003.

\bibitem{DS} R.J.~Duffin, A.C.~Schaeffer, \emph{A class of nonharmonic Fourier series}, Trans. Amer. Math. Soc., 72 (1952), 341--366.

\bibitem{FZ} S. Fomin, A. Zelevinsky, \emph{Total positivity: tests and parametrizations}, Math. Intelligencer 22 (2000), 23--33.

\bibitem{GP} M. Gasca, J.M. Pe\~{n}a, \emph{Total positivity and Neville elimination}, Lin. Alg. Appl., 165 (1992), 25--44.

\bibitem{GKK} V.K. Goyal, J. Kova\v cevi\' c, J.A. Kelner, \emph{Quantizied frame expansions with erasures}, Appl. Comp. Harm. Anal., 10 (2001), 203--233.

\bibitem{HL} D. Han, D.R. Larson, \emph{Frames, bases and group representations}, Memoirs Amer. Math. Soc., 697 (2000), 1--94.

\bibitem{HS} D. Han, W. Sun, \emph{Reconstruction  of signals from frame coefficients with erasures at unknown locations},
IEEE Trans. Inform. Theory 60 (2014), no. 7, 4013--4025.

\bibitem{HP} R. Holmes, V.I. Paulsen, \emph{Optimal frames for erasures}, Lin. Alg. Appl., 377 (2004), 31--51.

\bibitem{H} J. Holub, \emph{Pre-frame operators, Besselian frames and near-Riesz bases in Hilbert spaces}, Proc. Amer. Math. Soc.,
122 (1994), 779--785.


\bibitem{KC} J. Kova\v cevi\' c, A. Chebira, \emph{Life beyond bases: The advent of frames}, IEEE Signal Process. Mag., 24(4) (2007), 86--104.

\bibitem{LS} D. Larson, S. Scholze, \emph{Signal reconstruction from frame and sampling erasures},
J. Fourier Anal. Appl. 21 (2015), no. 5, 1146--1167.

\bibitem{LH} J. Leng, D. Han, \emph{Optimal dual frames for erasures II}, Lin. Alg. Appl., 435 (2011), 1464--1472.

\bibitem{PK}  M. P\"{u}schel, J. Kova\v cevi\' c, \emph{Real, tight frames with maximal robustness to erasures}, Proc. Data Compr.
Conf., (2005), 63--72.


\end{thebibliography}
\end{document}